\documentclass[reqno,twoside]{amsart}
\usepackage{amsmath}
\usepackage{amsfonts}
\usepackage{amssymb}
\usepackage{enumerate}
\usepackage{a4wide,amsmath,amssymb,latexsym,amsthm}

\numberwithin{equation}{section}
\usepackage{mathrsfs,mathtools,epic,bm}
\usepackage{hyperref}
\hypersetup{colorlinks=true,linkcolor=blue,filecolor=mangeta,urlcolor= cyan}

\newtheorem{theorem}{Theorem}[section]
\newtheorem{proposition}[theorem]{Proposition}
\newtheorem{lemma}[theorem]{Lemma}
\newtheorem{corollary}[theorem]{Corollary}
\theoremstyle{definition}
\newtheorem{definition}[theorem]{Definition}

\newtheorem{remark}[theorem]{Remark}

\newtheorem{assumption}[theorem]{Assumption}

\newcommand\abs[1]{\left\lvert#1\right\rvert}

\numberwithin{equation}{section}

\def \dis {\displaystyle}

\def \R {\mathbb{R}}
\def \X {\mathbb{X}}

\def \C {\mathcal{C}}

\def \C {{\mathcal C}}

\def \U {{\mathcal U}}

\def \hvarphi \hat{\varphi}
\def \dx{\mathrm{d}x}
\def \dxp{\mathrm{d}x'}
\def \dt{\mathrm{d}t}
\def \ds{\mathrm{d}s}

\def \dq {\mathrm{d}x \, \mathrm{d}t}
\def \d {\mathrm{d}}

\def \R {\mathbb{R}}
\def \X {\mathbb{X}}

\def \C {\mathcal{C}}

\def \X {\mathbb{ X}}

\def \U {\mathcal{U}}
\def \C {{\mathcal C}}

\keywords{Optimal control, nonlocal nonlinearity, first- and second- order optimality conditions}

\begin{document}
	\title[Optimal control of a parabolic equation with a nonlocal nonlinearity]{Optimal control of a parabolic equation with a nonlocal nonlinearity}
\author{Cyrille Kenne }
\address{Cyrille Kenne, Laboratoire L3MA, UFR STE et IUT, Universit\'e
	des Antilles, 97275 Schoelcher, Martinique}
\email{kenne853@gmail.com}
\author{Landry Djomegne }
\address{Landry Djomegne, University of Dschang, BP 67 Dschang, Cameroon, West region}
\email{landry.djomegne@gmail.com}
\author{Gis\`{e}le Mophou}
\address{Gis\`ele Mophou, Laboratoire LAMIA, D\'epartement de Math\'ematiques et Informatique, Universit\'{e} des Antilles, Campus Fouillole, 97159 Pointe-\`a-Pitre,(FWI), Guadeloupe, Laboratoire  MAINEGE, Universit\'e Ouaga 3S, 06 BP 10347 Ouagadougou 06, Burkina Faso.}
\email{gisele.mophou@univ-antilles.fr}
	\date{\today}

	\begin{abstract}
		This paper proposes an optimal control problem for a parabolic equation with a nonlocal nonlinearity. The system is described by a parabolic equation involving a nonlinear term that depends on the solution and its integral over the domain. We prove the existence and uniqueness of the solution to the system and the boundedness of the solution.  Regularity results for the control-to-state operator, the cost functional and the adjoint state are also proved. We show the existence of optimal solutions and derive first-order necessary optimality conditions. In addition, second-order necessary and sufficient conditions for optimality are established.
	\end{abstract}
	\maketitle	
	
\section{Introduction}
Let $\Omega\subset \R^n$ ($n\geq 1$) be an open bounded set, with regular boundary $\partial \Omega$. Let  $\omega$ is a nonempty subset of $\Omega$. We set $Q:=\Omega\times (0,T)$, $\Sigma=\partial\Omega\times (0,T)$ and $\omega_T =\omega\times (0,T )$, where $T>0$.

In this paper, we investigate the following control problem:
\begin{align}\label{opt}
	\dis \min_{v\in {\mathcal{U}_{ad}}}J(y,v),
\end{align}
with
\begin{align}\label{opt1}
	\dis J(y,v):=\frac 12\int_Q \abs{y-y_d}^2\,\dq+\frac{\mu}{2}\int_{\omega_T} \abs{v}^2\,\dq,
\end{align}
where   $y_d\in L^\infty(Q),$  $\mu>0$ and the set of admissible controls is given by	
\begin{equation}\label{defuad}
	\mathcal{U}_{ad}:=\left\{w\in L^\infty(\omega_T): \alpha\leq w\leq \beta,\, \alpha,\beta\in \R,\, \beta>\alpha\right\}.
\end{equation}
The state $y:=y(v)$ satisfies the following parabolic equation with a nonlinear term, nonlocal in space:
\begin{equation}\label{model}
	\left\{
	\begin{array}{rllll}
		\dis y_{t}-\Delta y+a\left(l(y(t))\right)y &=&v\chi_{\omega}& \mbox{in}& Q,\\
		\dis  \partial_\nu y&=&0& \mbox{in}& \Sigma, \\
		\dis  y(0,\cdot)&=&y^0 &\mbox{in}&\Omega.
	\end{array}
	\right.
\end{equation}
In the above system, $\partial_{\nu}$ is the usual normal derivative, $v\in L^2(\omega_T)$ is the control function , $\chi_{\omega}$  denotes the characteristic function of the control set $\omega$ and $y^0\in L^2(\Omega)$ is the initial data. The function $l$ is defined by
\begin{equation}\label{defl}
	l(y(t))=\int_{\Omega}y(t,x')\,\dxp.
\end{equation}

Recently, nonlocal problems have become increasingly considered by researchers. They have been
used for many application, for instance, in population dynamics \cite{chipot2000, chipot1999, correa2010,furter1989}, biology \cite{calsina1995, szymanska2009},  and physics \cite{chafee1981,lacey1995}, so they are of great importance and interest.  In the case of population dynamics, the most straightforward way of introducing non-local effects is by considering the crowding effect \cite{furter1989}. In this case, the state $y$ of \eqref{model} can be interpreted as a density of population and the system \eqref{model} describes the evolution of this population according to the fact that $v$ is a supply and $a\left(l(y(t))\right)$ proportional to the density of population by a factor $a$ represents the density of death or extinction of the population at stake \cite{chang2003a}. The term $l$ defined in \eqref{defl} represents the total population at any time $t$. As an example, we mention the population of bacteria in which the price of life in crowded regions could drop dramatically. \par
We make the following assumptions on the real function $a$:
\begin{assumption}\label{ass1}
	$ $
	\begin{itemize}
		\item The mortality rate of species increases with density. That is $r\mapsto a(r)$ is a non-decreasing function.
		\item The function $a$ is such that
		\begin{equation}\label{as1}
			\begin{array}{llll}
				\dis a\in \mathcal{C}^2(\R) \;\; \text{and} \;\; \exists \ a_0,a_1,M>0 \;\;\text{such that}\, \, 0<a_0\leq a(r)\leq a_1 \;\;\text{and}\\
				|a'(r)|+|a''(r)|\leq M \;\; \text{for all}\; r\in \R.
			\end{array}
		\end{equation}
	\end{itemize}
\end{assumption}
\begin{remark} We observe that the assumption $0 <  a_0 \leq a(r) \leq a_1$ for all $r\in \R$ is not too restrictive. Indeed, one can replace this assumption by: there exists $\alpha>0$ such that 
	\begin{equation*}
		|a(r)|\leq \alpha \;\;\text{for all}\;\; r\in \R,
	\end{equation*}
	and then make the classical change of variables $z:=e^{-kt}y$ in \eqref{model} for some $k>0$ to obtain:
	$$ \dis z_{t}-\Delta z+\left[k+a\left(\int_{\Omega}e^{kt}z(\cdot,x)\,dx\right)\right]z =e^{-kt}v\chi_{\omega}.$$
	Consequently, by choosing $k$ sufficiently large, we can always assume that this assumption holds.
\end{remark}

Optimal control problems have been widely studied in various fields of science \cite{anita2000,anicta2011, lenhart2007} and engineering \cite{ben2010optimal, geering2007, yang1975} due to their practical applications in designing optimal control strategies for complex systems. Nonlinear optimal control problems are particular challenging, since the associated objective function can be non convex, leading to multiple local minima. In addition, the first order necessary optimality conditions should be complemented by a second order analysis. 

Second-order analysis for nonlinear partial differential equations has been studied by several authors, see e.g. \cite{antil2020op,bayen2014, bonifacius2018,casas2018analysis, casas2020, casas2012, casas2016second,hoppe2022} and the references therein. This type of analysis is used to ensure the stability of optimal solutions with respect to perturbations of problems such as finite element discretizations \cite{casas2015a}. Regarding the second-order analysis for optimal control problems such as \eqref{opt}-\eqref{model}, we refer to \cite{casas2018} where the authors studied the measure control of a semilinear equation with a nonlocal time delay, and to the very recent paper \cite{casas2023} where the authors investigated the control of a parabolic equation with memory.

In this paper, we consider an optimal control problem for a parabolic equation with a nonlocal nonlinearity. The system is described by a partial differential equation involving a nonlinear term that depends on the solution and its integral over the domain. The nonlocal nonlinearity in the system arises due to the presence of the spatial integral operator, which makes the system more complex and challenging to analyze. Let us point some of them:
\begin{itemize}
	\item In general, the maximum principle may not hold for nonlocal equations such as \eqref{model}. The reason for this is that nonlocal equations involve integral operators that account for the effect of the entire domain, rather than just the local behavior at a given point. Hence, we carefully establish the boundedness of the solution to \eqref{model} in theorems \ref{theoremexistence0} and \ref{theorembound}.
	\item Unlike in other types of problems, the adjoint system \eqref{adjoint} is not of the same type as \eqref{model}. Consequently, the well-posedness of the adjoint problem and the boundedness of its solution cannot be deduced from the well-posedness of \eqref{model} and the boundedness of its solution. Therefore, it is necessary to establish these results separately, as done in Proposition \ref{existenceadjoint} and theorems \ref{theoremexistence} and \ref{theorembound0}.
\end{itemize}

Our objective is to minimize the cost functional \eqref{opt1},  including the difference between the solution and a desired target, as well as the $L^2$-norm of the control. We aim to prove the existence and uniqueness of the solution to the system \eqref{model} and derive necessary optimality conditions for the control problem \eqref{opt}-\eqref{model}. 

The rest of this paper is organized as follows. In Section \ref{wellposedness}, we give some notations and definitions of functional spaces and their associated norms for the need of this work; we prove some results on the existence, the uniqueness and the boundedness of the weak solution to \eqref{model}. In Section \ref{control}, we prove the existence of  optimal solutions to \eqref{opt}-\eqref{model}. First and second-order conditions for optimality are derived in Section  \ref{sec-5}. Finally, further directions are pointed out in Section \ref{conclusion}.

\

\section{Preliminary results\label{wellposedness}}
We start this section by introducing some functional spaces needed in the sequel.

Let $(H^{1}(\Omega))^\star$ be the topological dual space of $H^1(\Omega)$ and we denote by $\left\langle \cdot,\cdot\right\rangle$ the dual product between $H^1(\Omega)$ and $(H^{1}(\Omega))^\star$.
If we set
\begin{equation}\label{defWTA}
	W(0,T)= \left\{\rho \in L^2(0,T;H^1(\Omega)); \rho_t\in L^2\left(0,T;(H^{1}(\Omega))^\star\right)\right\},
\end{equation}
then $W(0,T)$ endowed with the norm
\begin{equation}\label{}
	\|\rho\|^2_{	W(0,T)}=\|\rho\|^2_{L^2(0,T;H^1(\Omega))}+\|\rho_t\|^2_{L^2\left(0,T;(H^{1}(\Omega))^\star\right)}
\end{equation}
is a Hilbert space. Moreover, we have that the following embedding is continuous
\begin{equation}\label{contWTA}
	W(0,T)\subset C([0,T];L^2(\Omega)),
\end{equation}
and the following embedding
\begin{equation}\label{comp}
	W(0,T)\subset L^2(Q)
\end{equation}
is compact.
The following result is classical and is obtained directly  from the Gagliardo-Niremberg Theorem (see e.g \cite{ breitenbach2019, nirenberg2011}).
\begin{lemma}\label{lemmaimpt}
	Let $y\in W(0,T)$. Then, $y\in L^{\sigma}(Q)$, with $\sigma=\frac{2(n+2)}{n}$. In addition, there exists a constant $C=C(n)$ such that the following estimate
	\begin{equation}\label{ineq100}
		\int_Q|y|^\sigma\,\dq\leq C\|y\|^{\frac{4}{n}}_{L^\infty(0,T;L^{2}(\Omega))}\int_{Q}|\nabla y|^2\,\dq
	\end{equation}
	holds.
\end{lemma}
We will also need the following result for the proof of the boundedness of solutions to \eqref{model}.
\begin{theorem}\cite[Lemma 4.1.1]{wu2006}\label{thmimp}
	Let $\varphi$ be a non-negative and non-increasing function on $[k_0,\infty)$ such that
	\begin{equation}\label{ineq2}
		\varphi(m)\leq \left(\frac{C}{m-k}\right)^{\alpha}(\varphi(k))^{\beta},\;\;\;\; \forall m>k> k_0,	
	\end{equation}
	where $C$, $\alpha, \beta$ are positive constants with $\beta>1$. Then
	$\varphi(m)=0$ for all $m\geq k_0+d$, where $d:=C2^{\frac{\beta}{\beta-1}}(\varphi(k_0))^{\frac{\beta-1}{\alpha}}$.
\end{theorem}

\section{Existence results}\label{existence}
In this subsection, we prove the existence and uniqueness of a class of parabolic systems. Let $y^0\in L^2(\Omega)$, $\tilde{a}_0\in L^\infty(Q)$ and $f\in L^2(Q)$. We consider the following problem:
\begin{equation}\label{general}
	\left\{
	\begin{array}{rllll}
		\dis y_t-\Delta y+a(l(y(t)))y+\tilde{a}_0\int_{\Omega}y(t,\sigma)\d\sigma &=& f &\mbox{in}& Q,\\
		\dis \partial_{\nu}y&=&0  &\mbox{in}& \Sigma,\\
		y(0,\cdot)&=& y^0 &\mbox{in}& \Omega.
	\end{array}
	\right.
\end{equation}

We define the weak solution to the system \eqref{general} as follows.
\begin{definition}\label{weakgen}
	Let $f\in L^2(Q)$, $\tilde{a}_0\in L^\infty(Q)$ and $y^{0}\in L^2(\Omega)$. Let $a(\cdot)$  satisfying the Assumption \ref{ass1}.  We  say that a function
	$y\in L^2(0,T;H^1(\Omega))$ is a weak solution to \eqref{general}, if the following equality
	\begin{equation}\label{weakgene}
		\begin{array}{lll}
			\dis -\int_{0}^{T} \left\langle \phi_t,y\right\rangle \, \dt+ \int_Q \nabla \phi\nabla y\, \dq+\int_Q a\left(l(y(t))\right) \phi y\, \dq+\int_Q\left(\tilde{a}_0 \dis \int_{\Omega}y(t,\sigma)\d\sigma\right)\phi\, \dq\\
			\qquad =
			\dis \int_\Omega y^0\,\phi(0,\cdot)\, \dx+\int_{Q}f\phi\,\dq ,
		\end{array}
	\end{equation}
	holds for every
	\begin{equation}\label{defH}
		\phi \in H(Q):=\{\varphi\in W(0,T): \varphi(T,\cdot)=0\}.
	\end{equation}
\end{definition}
In the rest of the paper, we denote by $C$ a positive constant that may vary line to line.
We have the following result.
%

\begin{theorem}\label{theoremexistence}
	Let $f\in L^2(Q)$, $\tilde{a}_0\in L^\infty(Q)$ such that $\|\tilde{a}_0\|_{L^\infty(Q)}\leq \tilde{C}$, for some $\tilde{C}>0$ and $y^{0}\in L^2(\Omega)$. Let $a(\cdot)$ satisfying the Assumption \ref{ass1}.  Then, there exists a unique weak solution $y \in W(0,T)$ to \eqref{general} in the sense of Definition~\ref{weakgen}.
	In addition, there exists a constant $C:=C(|\Omega|,a_0,\tilde{C},T)>0$ such that
	\begin{equation}\label{estimationgen}
		\|y\|_{\C([0,T];L^2(\Omega))}+ \|y\|_{L^2(0,T;H^1(\Omega))} \leq C\left(\|y^0\|_{L^2(\Omega)}+\|f\|_{L^2(Q)}\right)
	\end{equation}
	and
	\begin{equation}\label{estimationge}
		\|y\|_{W(0,T)}\leq C\left(\|y^0\|_{L^2(\Omega)}+\|f\|_{L^2(Q)}\right).
	\end{equation}
\end{theorem}
\begin{proof} We proceed in two steps.
	
	\noindent \textbf{Step 1.} We prove the existence of $y$ solution to \eqref{general}.
	The proof uses the Schauder fixed point theorem.
	Let us introduce the following convex subset of $L^2(Q)$
	$$B:=\left\{z\in L^2(Q):\|z\|_{L^2(Q)}\leq C_0\right\},$$
	where $\dis C_0:=\frac{e^{\tilde{C}|\Omega|T}}{\sqrt{\min(a_0,1)}}\left(\|y^0\|_{L^2(\Omega)}+\frac{1}{\sqrt{\min(a_0,1)}}\|f\|_{L^2(Q)}\right).$
	Let $z\in B$, we set
	\begin{equation}\label{defF}
		y=F(z)
	\end{equation}
	and consider the following linear problem
	\begin{equation}\label{generalin}
		\left\{
		\begin{array}{rllll}
			\dis y_t-\Delta y+a(l(z(t)))y+\tilde{a}_0\int_{\Omega}y(t,\sigma)\d \sigma &=& f &\mbox{in}& Q,\\
			\dis \partial_{\nu}y&=&0  &\mbox{in}& \Sigma,\\
			y(0,\cdot)&=& y^0 &\mbox{in}& \Omega.
		\end{array}
		\right.
	\end{equation}
	Then using Corollary \ref{corollaryexistence}, we have that \eqref{generalin} has a unique weak solution $y\in W(0,T)$. Moreover, there is a constant $C:=C(|\Omega|,a_0,\|y^0\|_{L^2(\Omega)}, \|f\|_{L^2(Q)})>0$ independent of $z$ such that
	\begin{equation}\label{est}
		\begin{array}{lll}
			\dis \|y\|_{L^2(Q)}\leq C_0,\\
			\dis \|y\|_{W(0,T)}\leq C.
		\end{array}
	\end{equation}
	
	Finally, from the first equation of \eqref{est}, we deduce that $F(B)\subset B$. In addition, from the last two equations of \eqref{est}, we deduce that $y\in W(0,T)$. Since this latter space is relatively compact in $L^2(Q)$, it follows that $F(B)$ is relatively compact in $B$. To conclude, it remains to show that $F$ is continuous from $B$ to $B$. We proceed as in \cite{chipot2000}.
	Let $(z_k)_k\in B$ be a sequence such that $z_k\to z$ in $L^2(Q)$. Set $y_k:=F(z_k)$, the solution to \eqref{generalin} corresponding to $z=z_k$. Then, from \eqref{est}, we can deduce that there exists  $y\in W(0,T)$ and a subsequence of $(y_k)_k$ always denoted by $(y_k)_k$ such that
	\begin{equation}\label{conv}
		\begin{array}{lll}
			\dis  y_k\rightharpoonup y   \text{ weakly in }  W(0,T).
		\end{array}
	\end{equation}
	Thanks the compact embedding $W(0,T)\hookrightarrow L^2(Q)$, we deduce from \eqref{conv} that
	\begin{equation}\label{conv1}
		y_k\to y   \text{ strongly in }  L^2(Q).
	\end{equation}
	Next, we claim that as $k\to \infty$,
	\begin{equation}\label{conv2}
		a(l(z_k(\cdot)))\to a(l(z(\cdot)))   \text{ strongly in }  L^2(Q).
	\end{equation}	
	
	Indeed, we have using the mean value theorem and \eqref{as1} that
	
	$$
	\begin{array}{rllll}
		\dis 	\|a(l(z_k(\cdot)))- a(l(z(\cdot)))\|^2_{L^2(Q)}&=& \dis \int_0^T\int_\Omega|a(l(z_k(t)))- a(l(z(t)))|^2\,dx\, \dt\\
		&\leq& M^2\dis \int_0^T\int_\Omega|l(z_k(t))- l(z(t))|^2\,dx\, \dt,
	\end{array}
	$$
	which in view of \eqref{defl} and the Cauchy-Schwarz inequality gives
	$$
	\begin{array}{rllll}
		\dis 	\|a(l(z_k(\cdot)))- a(l(z(\cdot)))\|^2_{L^2(Q)}&\leq& M^2\dis \int_0^T\int_\Omega\left[\left(\int_\Omega dx^\prime\right)^{1/2}\left(\int_\Omega|z_k(t,x^\prime)- z(t,x^\prime)|^2dx^\prime\right)^{1/2}\right]^2\,dx\, \dt\\
		&\leq& M^2 |\Omega|^2\dis \int_{Q}|z_k- z|^2\,\dq.
	\end{array}
	$$
	
	Taking the limit as $k\to \infty$ in this latter inequality leads us to \eqref{conv2}.
	
	Now, if we multiply the first equation in \eqref{generalin} by $\phi\in H(Q)$ and we integrate by parts over $Q$, we obtain
	\begin{equation*}
		\begin{array}{lll}
			\dis -\int_{0}^{T} \left\langle \phi_t,y_k\right\rangle \, \dt + \int_Q \nabla \phi\nabla y_k\, \dq+\int_Q a\left(l(z_k(t))\right) \phi y_k\, \dq+\int_Q\left(\tilde{a}_0 \dis \int_{\Omega}y_k(t,\sigma)\d\sigma\right)\phi\, \dq\\
			\qquad =
			\dis \int_\Omega y^0\,\phi(0,\cdot)\, \dx+\int_{\omega_T}f\phi\,\dq \;\;\;\text{for every}\;\;\; \phi \in H(Q).
		\end{array}
	\end{equation*}
	Passing to the limit in this latter identity, while using the convergences \eqref{conv}-\eqref{conv2} and the Lebesgue dominated convergence Theorem, we obtain
	\begin{equation}\label{2}
		\begin{array}{lll}
			\dis -\int_{0}^{T} \left\langle \phi_t,y\right\rangle \, \dt + \int_Q \nabla \phi\nabla y\, \dq+\int_Q a\left(l(z(t))\right) \phi y\, \dq+ \int_Q\left(\tilde{a}_0 \dis \int_{\Omega}y(t,\sigma)\d\sigma\right)\phi\, \dq\\
			\qquad =
			\dis \int_\Omega y^0\,\phi(0,\cdot)\, \dx+\int_{\omega_T}f\phi\,\dq \;\;\;\text{for every}\;\;\; \phi \in H(Q).
		\end{array}
	\end{equation}
	Therefore,  we can
	deduce that $y$ is a weak solution to \eqref{generalin}. Since this latter problem has a unique solution (see Corollary \ref{corollaryexistence}),  we conclude that  $y_k=F(z_k)$ converges strongly to $y=F(z)$ in $L^2(Q)$. Hence, the continuity of $F$ is established. By applying the Schauder fixed point theorem in $B$, we deduce that there exists a unique $y\in B$ such that $y=F(y)$. 
	Using again Corollary \ref{corollaryexistence}, with $z=y\in L^2(Q)$ we have that the weak solution $y=F(y)$ of \eqref{general} belongs to $W(0,T)$ and satisfies the estimates \eqref{estimationgen}-\eqref{estimationge}.
	
	\noindent \textbf{Step 2.} We prove the uniqueness of $y$ solution to \eqref{general}.\\
	Assume that there exist $y_1$ and $y_2$ solutions to \eqref{general} with the same right hand side $f$ and initial datum $y^0$.  Set $z:=e^{-rt}(y_1-y_2)$, with some $r>0$. Then  $z$ satisfies the following system
	\begin{equation}\label{aa1}
		\left\{
		\begin{array}{rllll}
			\dis z_{t}-\Delta z+a\left(l(y_1(t))\right)z+\tilde{a}_0\int_{\Omega}z(t,\sigma)\d\sigma +rz &=&-e^{-rt}\left[ a\left(l(y_1(t))\right)-a\left(l(y_2(t))\right)\right]y_2& \mbox{in}& Q,\\
			\dis  \partial_\nu z&=&0& \mbox{in}& \Sigma, \\
			\dis  z(0,\cdot)&=&0 &\mbox{in}&\Omega.
		\end{array}
		\right.
	\end{equation}
	We remark that for every $\phi\in W(0,T)$, we have
	\begin{equation}\label{ine1}
		\left|\int_Q\left(\tilde{a}_0 \dis \int_{\Omega}z(t,\sigma)\d\sigma\right)\phi\, \dq\right|\leq \tilde{C}|\Omega|\|z\|_{L^2(Q)}\|\phi\|_{L^2(Q)}.
	\end{equation}
	Multiplying the first equation in \eqref{aa1} by $z(t)$ for $t\in (0,T)$, integrating by parts over $\Omega$, using \eqref{as1} and \eqref{ine1}, we arrive to
	\begin{equation}\label{ajout01}
		\begin{array}{rllll}
			\dis \frac{1}{2}\frac{d}{dt}\|z(t)\|^2_{L^2(\Omega)}+\|\nabla z(t)\|^2_{L^2(\Omega)}+ (r+a_0-\tilde{C}|\Omega|)\|z(t)\|^2_{L^2(\Omega)} \\
			\leq\left| a\left(l(y_1(t))\right)-a\left(l(y_2(t))\right)\right|\|y_2(t)\|_{L^2(\Omega)}\|z(t)\|_{L^2(\Omega)}.
		\end{array}
	\end{equation}
	Using the mean value theorem and thanks to  \eqref{as1}, we have that
	\begin{equation}\label{ajout02}
		\left| a\left(l(y_1(t))\right)-a\left(l(y_2(t))\right)\right|\leq Me^{rT}|\Omega|^{\frac{1}{2}}\|z(t)\|_{L^2(\Omega)}.
	\end{equation}
	Now, combining \eqref{estimationgen}, \eqref{ajout01} and \eqref{ajout02} and choosing $r=\tilde{C}|\Omega|$, we deduce that
	\begin{equation}\label{ajout03}
		\begin{array}{rllll}
			\dis \frac{1}{2}\frac{d}{dt}\|z(t)\|^2_{L^2(\Omega)}\leq  Me^{\tilde{C}|\Omega|T}C(|\Omega|,a_0,\|y^0\|_{L^2(\Omega)}, \|f\|_{L^2(Q)})|\Omega|^{\frac{1}{2}}\|z(t)\|^2_{L^2(\Omega)}.
		\end{array}
	\end{equation}
	Using the Gronwall inequality and since $z(0,\cdot)=0$, we arrive to $z=0$ in $Q$. Hence $y_1=y_2$. This completes the proof.
\end{proof}
%
By taking $\tilde{a}_0=0$ in the Theorem \ref{theoremexistence}, we obtain the following existence result for the problem \eqref{model}.
\begin{theorem}\label{theoremexistence0}
	Let  $v\in L^2(\omega_T)$ and $y^{0}\in L^2(\Omega)$. Let $a(\cdot)$ satisfying the Assumption \ref{ass1}.  Then, there exists a unique weak solution $y \in W(0,T)$ to \eqref{model} in the sense of Definition~\ref{weakgen}.
	In addition, there exists a constant $C>0$ such that
	\begin{equation}\label{estimation0}
		\|y\|_{\C([0,T];L^2(\Omega))}+ \|y\|_{L^2(0,T;H^1(\Omega))} \leq C\left(\|y^0\|_{L^2(\Omega)}+\|v\|_{L^2(\omega_T)}\right)
	\end{equation}
	and
	\begin{equation}\label{estimation00}
		\|y\|_{W(0,T)}\leq C\left(\|y^0\|_{L^2(\Omega)}+\|v\|_{L^2(\omega_T)}\right).
	\end{equation}
\end{theorem}
The next theorem shows that on additional assumptions, the weak solution to \eqref{general} are bounded.

\begin{theorem}\label{theorembound0}Let W(0,T) be defined as in \eqref{defWTA}.
	Let  $f\in L^\infty(Q)$, $\tilde{a}_0\in L^\infty(Q)$ such that $\|\tilde{a}_0\|_{L^\infty(Q)}\leq \tilde{C}$, for some $\tilde{C}>0$ and $y^{0}\in L^\infty(\Omega)$. Let $a(\cdot)$ satisfying the Assumption \ref{ass1}.  Then the unique weak solution $y \in W(0,T)$ to \eqref{general} belongs to $L^\infty(Q)$.
	In addition, there exists a constant $C=C(n,\tilde{C}, |\Omega|, T)>0$ such that
	\begin{equation}\label{estimationbound0}
		\|y\|_{L^\infty(Q)} \leq C\left(\|f\|_{L^\infty(\omega_T)}+\|y^0\|_{L^\infty(\Omega)}\right).
	\end{equation}
\end{theorem}
\begin{proof}
	We set $z=e^{-rt}y$ where $y$ is solution to \eqref{general} and for $r=\tilde{C}|\Omega|$. Then $z\in W(0,T)$ is the unique weak solution to
	\begin{equation}\label{general0}
		\left\{
		\begin{array}{rllll}
			\dis z_t-\Delta z+a(l(e^{rt}z(t)))z+\tilde{a}_0\int_{\Omega}z(t,\sigma)\d \sigma+rz &=& e^{-rt}f &\mbox{in}& Q,\\
			\dis \partial_{\nu}z&=&0  &\mbox{in}& \Sigma,\\
			z(0,\cdot)&=& y^0 &\mbox{in}& \Omega.
		\end{array}
		\right.
	\end{equation}
	Let $k$ be a real number such that $k>\|y^0\|_{L^\infty(\Omega)}$ and $t\in [0,T]$.  Then  for almost every $t\in (0,T)$, $z(t)-k\in H^1(\Omega)$ and so $(z-k)^+(t):=\max\{z(t)-k,0\}$ belongs to $H^1(\Omega)$.  For almost every $t\in (0,T)$, we set
	\begin{equation}\label{defA}
		\begin{array}{lll}
			\dis A(z(t), \phi(t))&=&\dis \int_{\Omega}\nabla z(t)\nabla \phi(t)\,\dx+ \int_{\Omega}a(l(e^{rt}z(t)))z(t)\phi(t)\,\dx\\
			&&\dis +\int_{\Omega}\left(\tilde{a}_0 \dis \int_{\Omega}z(t,\sigma)\d\sigma\right)\phi(t)\, \dx+ r\int_{\Omega}z(t)\phi(t)\,\dx.
		\end{array}
	\end{equation}
	Then, since $r\mapsto a(r)$ is non-decreasing and considering the chosen value of $r$, we can use \eqref{as1} and perform straightforward computations to obtain
	\begin{equation}\label{s1}
		\begin{array}{lll}
			\dis A((z-k)(t), (z-k)^+(t))\leq\dis A((z(t), (z-k)^+(t)).
		\end{array}
	\end{equation}
	If we multiplying the first equation in \eqref{general0} by $(z-k)^+(t)$ and integrate by parts over $\Omega$, we arrive to
	\begin{equation}\label{s2}
		\begin{array}{lll}
			\left\langle (z-k)_t(t),(z-k)^+(t)\right\rangle+\dis A((z(t), (z-k)^+(t))=\int_{\Omega}e^{-rt}f(t)(z-k)^+(t)\,\dx.
		\end{array}
	\end{equation}
	Combining \eqref{s1} and \eqref{s2}, it follows that
	$$
	\begin{array}{lll}
		\left\langle (z-k)_t(t),(z-k)^+(t)\right\rangle+\dis A((z-k)(t), (z-k)^+(t))\leq \int_{\Omega}e^{-rt}f(t)(z-k)^+(t)\,\dx,
	\end{array}
	$$
	which in view of the fact that
	$$
	\dis \left\langle (z-k)_t(t),(z-k)^+(t)\right\rangle=\frac{1}{2}\frac{d}{dt}\|(z-k)^+(t)\|^2_{L^2(\Omega)}
	$$
	gives
	\begin{equation}\label{s3}
		\begin{array}{lll}
			\dis 	\frac{1}{2}\frac{d}{dt}\|(z-k)^+(t)\|^2_{L^2(\Omega)}+\dis A((z-k)(t), (z-k)^+(t))\leq \int_{\Omega}e^{-rt}f(t)(z-k)^+(t)\,\dx.
		\end{array}
	\end{equation}
	Using \eqref{as1}, \eqref{ine1} and since $r=\tilde{C}|\Omega|$, we have that
	$$
	\begin{array}{lll}
		\dis A((z-k)(t), (z-k)^+(t))
		\geq\dis  \|\nabla(z-k)^+(t)\|^2_{L^2(\Omega)},
	\end{array}
	$$
	which combined with \eqref{s3} yields
	\begin{equation}\label{inter1}
		\frac{1}{2}\frac{d}{dt}\|(z-k)^+(t,\cdot)\|^2_{L^2(\Omega)}+\|\nabla(z-k)_+(t,\cdot)\|^2_{L^2(\Omega)}\leq \int_{\Omega}|f(t,\cdot)||(z-k)^+(t,\cdot)|\,\dx.
	\end{equation}
	Observing that $(z-k)^+(0):=max(z(0)-k,0)=0$ and  integrating  \eqref{inter1} on $(0,t)$, with $t\in [0,T]$, we obtain that
	\begin{equation}
		\frac{1}{2}\|(z-k)^+(t)\|^2_{L^2(\Omega)}+\int_{0}^{t}\|\nabla(z-k)^+(s)\|^2_{L^2(\Omega)}\,\ds\leq \int_{Q}|f||(z-k)^+|\,\dq.
	\end{equation}
	Hence,
	\begin{subequations}
		\begin{alignat}{11}
			\dis \|(z-k)^+\|^2_{L^\infty (0,T;L^2(\Omega))}\leq  2 \int_{Q}|f||(z-k)^+|\,\dq,\label{inter200}\\
			\|\nabla(z-k)^+\|^2_{L^2(Q)}\leq 2 \int_{Q}|f||(z-k)^+|\,\dq.\label{inter3}
		\end{alignat}
	\end{subequations}
	
	Now, using Lemma \ref{lemmaimpt}, we obtain that
	$$
	\begin{array}{lll}
		\dis\left(\int_{Q} |(z-k)^+|^{\frac{2(n+2)}{n}}\,\dq\right)^{\frac{n}{n+2}}
		\leq \left[ \|(z-k)^+\|^{\frac{4}{n}}_{L^\infty(0,T;L^2(\Omega))}\|\nabla(z-k)^+\|^{2}_{L^2(Q)}\right]^{\frac{n}{n+2}},
	\end{array}
	$$
	which in view of the fact that
	$$\dis\left[ \|(z-k)^+\|^{\frac{4}{n}}_{L^\infty(0,T;L^2(\Omega))}\|\nabla(z-k)^+\|^{2}_{L^2(Q)}\right]^{\frac{n}{n+2}}=
	\|(z-k)^+\|^{\frac{4}{n+2}}_{L^\infty(0,T;L^2(\Omega))}\|\nabla(z-k)^+\|^{\frac{2n}{n+2}}_{L^2(Q)}
	$$
	implies that
	\begin{equation}\label{ineq4}
		\begin{array}{lll}
			\dis\left(\int_{Q} |(z-k)^+|^{\frac{2(n+2)}{n}}\,\dq\right)^{\frac{n}{n+2}}
			\leq \|(z-k)^+\|^{\frac{4}{n+2}}_{L^\infty(0,T;L^2(\Omega))}\|\nabla(z-k)^+\|^{\frac{2n}{n+2}}_{L^2(Q)}.
		\end{array}
	\end{equation}
	On the other hand, using the Young's inequality with $p=\frac{n+2}{2}$ and $q=\frac{n+2}{n}$, we arrive to
	$$
	\begin{array}{lll}
		\dis \|(z-k)^+\|^{\frac{4}{n+2}}_{L^\infty(0,T;L^2(\Omega))}\|\nabla(z-k)^+\|^{\frac{2n}{n+2}}_{L^2(Q)}\leq\|(z-k)^+\|^2_{L^\infty(0,T;L^2(\Omega))}+\|\nabla(z-k)^+\|^2_{L^2(Q)},
	\end{array}
	$$
	which in view of \eqref{inter200} and \eqref{inter3} gives
	\begin{equation}\label{ineq3}
		\begin{array}{lll}
			\dis \|(z-k)^+\|^{\frac{4}{n+2}}_{L^\infty(0,T;L^2(\Omega))}\|\nabla(z-k)^+\|^{\frac{2n}{n+2}}_{L^2(Q)}
			\leq 4 \dis \int_{Q}|f||(z-k)^+|\,\dq.
		\end{array}
	\end{equation}
	Combining \eqref{ineq4}-\eqref{ineq3}, we arrive to
	\begin{equation}\label{ineq6}
		\int_{Q} |(z-k)^+|^{\frac{2(n+2)}{n}}\,\dq\leq 4^{\frac{n+2}{n}} \left(\int_{\omega_T}|f||(z-k)^+|\,\dq\right)^{\frac{n+2}{n}}.
	\end{equation}
	Next, we set $A_k:=\{(x,t)\in Q:z(x,t)>k\}$. Then $A_k$ is a measurable set. Moreover, we deduce from \eqref{ineq6} that
	\begin{equation}\label{ineq7}
		\int_{A_k} |(z-k)^+|^{\frac{2(n+2)}{n}}\,\dq\leq  4^{\frac{n+2}{n}} \left(\int_{A_k}|f||(z-k)^+|\,\dq\right)^{\frac{n+2}{n}}.
	\end{equation}
	Using the H\"{o}lder's inequality with $p=\frac{2(n+2)}{n+4}$ and $q=\frac{2(n+2)}{n}$ in the right hand side of \eqref{ineq7}, we obtain
	\begin{equation}\label{ineq8}
		\begin{array}{lll}
			\dis\left(\int_{A_k} |(z-k)^+|^{\frac{2(n+2)}{n}}\,\dq\right)^{1/2}
			\leq  4^{\frac{n+2}{n}}\|f\|^{\frac{2(n+2)}{n}}_{L^\infty(Q)}|A_k|^{\frac{n+4}{n}},
		\end{array}
	\end{equation}
	where $|A_k|$ denotes the measure of the set $A_k$.\\
	Let $m>k$, then $A_m\subseteq A_k$. Consequently,  for a.e. $(x,t)\in A_m$, we have  $z(x,t)-k>m-k$ and then,  $z-k=(z-k)^+$. Thus,
	\begin{equation}\label{ineq9}
		\begin{array}{lll}
			\dis\int_{A_k} |(z-k)^+|^{\frac{2(n+2)}{n}}\,\dq\geq (m-k)^{\frac{2(n+2)}{n}}|A_m|.
		\end{array}
	\end{equation}
	Combining \eqref{ineq8} and \eqref{ineq9}, we arrive to
	\begin{equation}\label{ineq10}
		\begin{array}{lll}
			\dis |A_m|\leq \left[\frac{4\|v\|^2_{L^\infty(\omega_T)}}{m-k}\right]^{\frac{2(n+2)}{n}}|A_k|^{\frac{2(n+4)}{n}}.
		\end{array}
	\end{equation}
	Using  Theorem \ref{thmimp} with $k_0=C(n)\left(\|f\|_{L^\infty(Q)}+\|y^0\|_{L^\infty(\Omega)}\right)$, we have $|A_{k_0}|=0$ and then we obtain that $|A_m|=0$ for any $m>C(n)\left(\|f\|_{L^\infty(Q)}+\|y^0\|_{L^\infty(\Omega)}\right)$. Therefore, the set $A_m$, where $z>C(n)\left(\|f\|_{L^\infty(Q)}+\|y^0\|_{L^\infty(\Omega)}\right)$ has measure zero. In the same way, by setting $(z+k)^-:=min(z+k,0)$ and  $A_k:=\{(x,t)\in Q:z(x,t)<k\}$ one show that the set $A_m$, where $z<-C(n)\left(\|f\|_{L^\infty(Q)}+\|y^0\|_{L^\infty(\Omega)}\right)$ has measure zero. Hence,
	\begin{equation*}
		\|z\|_{L^\infty(Q)}\leq C(n)\left(\|f\|_{L^\infty(Q)}+\|y^0\|_{L^\infty(\Omega)}\right).
	\end{equation*}
	Now, since $z=e^{-\tilde{C}|\Omega|t}y$, we arrive to
	\begin{equation*}
		\|y\|_{L^\infty(Q)}\leq C(n,\tilde{C}, |\Omega|, T)\left(\|f\|_{L^\infty(Q)}+\|y^0\|_{L^\infty(\Omega)}\right).
	\end{equation*}
	This completes the proof.	
\end{proof}
By setting $\tilde{a}_0=0$ in the Theorem \ref{theorembound0}, we deduce the following result.
\begin{theorem}\label{theorembound}Let W(0,T) be defined as in \eqref{defWTA}, $v\in L^\infty(\omega_T)$ and $y^{0}\in L^\infty(\Omega)$. Let $a(\cdot)$ satisfying the Assumption \ref{ass1}.  Then the unique weak solution $y \in W(0,T)$ to \eqref{model} belongs to $L^\infty(Q)$.
	In addition, there exists a constant $C=C(n)>0$ such that
	\begin{equation}\label{estimationbound}
		\|y\|_{L^\infty(Q)} \leq C\left(\|v\|_{L^\infty(\omega_T)}+\|y^0\|_{L^\infty(\Omega)}\right).
	\end{equation}
\end{theorem}

\section{Resolution of the optimization problem \label{control}}
We start this section by the following definition:
\begin{definition}\label{ctso}
We define the control-to-state mapping
$$
G:L^\infty(\omega_T)\to W(0,T)\cap L^\infty(Q),
$$
which associates to each $u\in L^\infty(\omega_T)$ the unique weak solution $y\in W(0,T)\cap L^\infty(Q)$ of \eqref{model}. Sometimes, we may write $G(u)$ to denote the state $y$ corresponding to the control $u$.
\end{definition}
Keeping this definition in mind, the optimal control problem \eqref{opt} can be rewritten as :
\begin{equation}\label{opt2}
\inf_{v\in \U_{ad}} J(G(v),v).
\end{equation}

We introduce the following notion of local  solutions.
\begin{definition}\label{defopt}
Let For $1 \leq p \leq +\infty$. We say that $u\in \U_{ad}$ is an $L^p$-local solution of \eqref{opt} if there exists $\varepsilon>0$ such that $J(G(u),u)\leq J(G(v),v)$ for every $v\in \U_{ad}\cap B^p_\varepsilon(u)$ where $ B^p_\varepsilon(u)=\{u\in L^p(\Omega):\|v-u\|_{L^p(\Omega)}\leq \varepsilon\}$. We say that $u$ is a strict local minimum of \eqref{opt} if the above inequality is strict whenever $v\neq u$.
\end{definition}

\subsection{Existence of optimal solutions}
In this subsection, we prove the existence of optimal solutions to \eqref{opt2}-\eqref{model}.	
\begin{theorem}\label{existcontrol1}
Let  $\mu>0$, $v\in \mathcal{U}_{ad}$ and $y^{0},y^d\in L^2(\Omega)$. Then  there exists at least a solution  $u\in \mathcal{U}_{ad}$ of the optimal control problem \eqref{opt2}-\eqref{model}.
\end{theorem}
\begin{proof}
Notice that $J(G(v),v)\geq 0$ for all $v\in \mathcal{U}_{ad}$. Let $\{(G(v_k),v_k)\}_k\subset W(0,T)\times\mathcal{U}_{ad}$ be a minimizing sequence such that
$$\lim_{k\to \infty}J(G(v_k),v_k)\to \min_{v\in\U}J(G(v),v). $$	
Then, there exists a constant $C_0>0$ such that
\begin{equation*}
	\|v_k\|_{L^2(\omega_T)}\leq C_0.
\end{equation*}
Since $y_k:=G(v_k)$ is the solution of \eqref{model} associated to  the control $v_k$, it follows from the latter inequality and \eqref{estimation00} that there is a positive constant $C$ independent of $n$ such that
\begin{equation}\label{b2}	
	\|y_k\|_{W(0,T)} \leq C.
\end{equation}
From the boundedness of $\U_{ad}$ in $L^\infty(\omega_T)$ and \eqref{b2}, we deduce the existence of  $u\in L^\infty(\omega_T)$ and $y\in W(0,T) $ such that up to a subsequence and as $k\to \infty$, we have
\begin{equation}\label{c1}
	v_k\rightharpoonup u   \text{ weakly-$\star$ } L^\infty(\omega_T)
\end{equation}
and
\begin{equation}\label{c2}
	y_k\rightharpoonup y   \text{ weakly in }  W(0,T).
\end{equation}
Using the compact embedding $W(0,T)\hookrightarrow L^2(Q)$, we deduce from \eqref{c2} that
\begin{equation}\label{c3}
	y_k\to y   \text{ strongly in }  L^2(Q).
\end{equation}

Next, we show as in \eqref{conv2} that as $k\to \infty$,
\begin{equation}\label{c4}
	a(l(y_k(\cdot)))\to a(l(y(\cdot)))   \text{ strongly in }  L^2(Q).
\end{equation}	

Finally, it remains to prove that $y=y(u)$, i.e. $y$ is the state associated to the control $u$. Recalling that $y_k:=G(v_k)$ is the solution of \eqref{model} associated to  the control $v_k$, we have
\begin{equation}\label{weaksoln}
	\begin{array}{lll}
		\dis -\int_{0}^{T} \left\langle \phi_t,y_k\right\rangle \, \dt + \int_Q \nabla \phi\nabla y_k\, \dq+\int_Q a\left(l(y_k(t))\right) \phi y_k\, \dq\\
		\qquad =
		\dis \int_\Omega y^0\,\phi(0,\cdot)\, \dx+\int_{\omega_T}v_k\phi\,\dq \;\;\;\text{for every}\;\;\; \phi \in W(0,T), \; \phi(T,\cdot)=0\;\; \text{in}\;\; \Omega.
	\end{array}
\end{equation}
Passing to the limit in this latter identity, while using the convergences \eqref{c1}-\eqref{c4}, we obtain
\begin{equation}\label{weaksol0}
	\begin{array}{lll}
		\dis -\int_{0}^{T} \left\langle \phi_t,y\right\rangle \, \dt + \int_Q \nabla \phi\nabla y\, \dq+\int_Q a\left(l(y(t))\right) \phi y\, \dq\\
		\qquad =
		\dis \int_\Omega y^0\,\phi(0,\cdot)\, \dx+\int_{\omega_T}u\phi\,\dq \;\;\;\text{for every}\;\;\; \phi \in W(0,T),\; \phi(T,\cdot)=0\;\; \text{in}\;\; \Omega.
	\end{array}
\end{equation}
Hence, $y=y(u)$. In addition, since $\U_{ad}$ is a closed convex subset of $L^\infty(\omega_T)$, we have that $\U_{ad}$ is weakly closed and \eqref{c1} implies
\begin{equation}\label{ajout5}
	u\in \U_{ad}.
\end{equation}
Moreover, using \eqref{c3}, \eqref{ajout5}, and the lower semi-continuity of the cost functional $J$, it follows that
\begin{eqnarray*}
	J(G(u),u)&\leq &\liminf_{k\to \infty}J(G(v_k),v_k)=\inf_{v\in\U}J(G(v),v).
\end{eqnarray*}	
This completes the proof.		
\end{proof}

\section{Optimality conditions}\label{sec-5}
In this section, we establish the first and second order optimality conditions for the problem \eqref{opt}-\eqref{model}. 	
Before going further, we need some regularity results for the control-to-state operator.
In the rest of the paper, we assume that $y^0\in L^\infty(\Omega)$, so that the weak solution $y$ of \eqref{model} belongs to $\mathbb{X}:=W(0,T)\cap L^\infty(Q)$ (see Theorem \ref{theorembound}).
\subsection{Regularity results for the control-to-state operator}
Here, we derive some regularity results for the control-to-state operator. Let us introduce the vector space
\begin{equation}\label{space1}
\mathbb{Y}=\left\{z\in \mathbb{ X}: \dis z_{t}-\Delta z\in L^\infty(Q) \right\}.
\end{equation}
Then,  $\mathbb{ Y}$ endowed with the graph norm
\begin{equation}\label{norm1}
\|z\|_{\mathbb{ Y}}=\|z\|_{\mathbb{ X}}+\|z_{t}-\Delta z\|_{L^{\infty}(Q)},\quad\forall z\in \mathbb{Y}
\end{equation}
is a Banach space.
Let us introduce the mapping
\begin{equation}\label{defG}
(y, u)\mapsto\mathcal{G}(y,u):= (\dis y_{t}-\Delta y+a\left(l(y(t))\right)y -u\chi_{\omega},y(0)-y^0)
\end{equation}
from $ \mathbb{ Y}\times L^\infty(\omega_T)\to L^\infty(Q)\times L^\infty(\Omega).$ The mapping $\mathcal{G}$ is well defined and  the state equation $y$ solution of \eqref{model} can be viewed as the equation
\begin{equation}\label{eqG}
\mathcal{G}(y,u)=0.
\end{equation}

\begin{lemma}\label{lemmeG}
Under the Assumption \ref{ass1}, the mapping $\mathcal{G}$ is of class $\mathcal{C}^{2}$. Moreover, the control-to-state mapping $G:L^\infty(\omega_T)\to\mathbb{ X}, u\mapsto y$ is also of class $\mathcal{C}^{2}$.
\end{lemma}
\begin{proof}
Under the Assumption \ref{ass1}, $a\in \mathcal{C}^2(\R)$ so we deduce that the first component  of $\mathcal{G}$ is of class $\mathcal{C}^2$. The second component of $\mathcal{G}$ is clearly of class $\mathcal{C}^{\infty}$. Therefore, the mapping $\mathcal{G}$ is of class $\mathcal{C}^{2}$.
Moreover,
$$\partial_y\mathcal{G}(y,u)\varphi=\left(\varphi_t+a'(l(y(t)))y\int_{\Omega}\varphi(t,\sigma)\d \sigma-\Delta \varphi+a(l(y(t)))\varphi,\varphi(0)\right).$$
Let $\varphi^0\in L^\infty(\Omega)$ and $f\in L^\infty(Q)$. We consider the linear problem
\begin{equation}\label{der}
	\left\{
	\begin{array}{rllll}
		\dis \varphi_t+a'(l(y(t)))y\int_{\Omega}\varphi(t,\sigma)\d\sigma-\Delta \varphi+a(l(y(t)))\varphi &=& f &\mbox{in}& Q,\\
		\dis \partial_{\nu}\varphi&=&0  &\mbox{in}& \Sigma,\\
		\varphi(0,\cdot)&=& \varphi^0 &\mbox{in}& \Omega.
	\end{array}
	\right.
\end{equation}
Then using theorems \ref{theoremexistence} and \ref{theorembound0} with $\tilde{a}_0=a'(l(y(t)))y\in L^\infty(Q)$ and thanks to \eqref{as1} and \eqref{estimationbound}, we deduce that \eqref{der} has a unique weak solution  $\varphi:=\varphi(\varphi^0,f)$ in $\mathbb{X}$ depending continuously on $\varphi^0\in L^\infty(\Omega)$ and on $f\in L^\infty(Q)$.  Moreover, since $y\in L^\infty(Q)$ and thanks to Assumption \ref{ass1}, we have $ \varphi_t-\Delta \varphi=f-a'(l(y(t)))y\int_{\Omega}\varphi(t,\sigma)\d\sigma-a(l(y(t)))\varphi\in L^\infty(Q)$. Therefore, $\varphi\in \mathbb{ Y}$. Hence,
$\dis \partial_y\mathcal{G}(y,u)$ define an isomorphism from $\mathbb{ Y}$ to $L^\infty(Q)\times L^\infty(\Omega)$. Using the Implicit Function Theorem, we deduce that $\mathcal{G}(y,u)=0$ implicitly define the control-to-state operator $G:u\mapsto y$ which is itself of class  $\mathcal{C}^{\infty}$.
\end{proof}
We have the following result.
\begin{proposition}\label{differentiability}
Let $u,v, w\in L^\infty(\omega_T)$.	Under the assumptions of Lemma \ref{lemmeG}, the first derivative of the control-to-state operator $G'(u)$ is given by $G'(u)v=z_v$, where $z_v$ is the unique weak solution $z\in \mathbb{X}$ to
\begin{equation}\label{diff1}
	\left\{
	\begin{array}{rllll}
		\dis z_t-\Delta z+a'(l(y(t)))y\int_{\Omega}z(\cdot,\sigma)\d\sigma+a(l(y(t)))z &=& v\chi_{\omega} &\mbox{in}& Q,\\
		\dis \partial_{\nu}z&=&0  &\mbox{in}& \Sigma,\\
		z(0,\cdot)&=& 0&\mbox{in}& \Omega.
	\end{array}
	\right.
\end{equation}
The second derivative $G''(u)$ is given by $G''(u)(v,w)=z_{vw}$, where $z_{vw}$ is the unique weak solution $z\in \mathbb{X}$	to
\begin{equation}\label{diff2}
	\left\{
	\begin{array}{rllll}
		\dis z_t-\Delta z+a(l(y(t)))z+a'(l(y(t)))y\int_{\Omega}z(\cdot,\sigma)\d\sigma &=&F &\mbox{in}& Q,\\
		\dis \partial_{\nu}z&=&0  &\mbox{in}& \Sigma,\\
		z(0,\cdot)&=& 0&\mbox{in}& \Omega,
	\end{array}
	\right.
\end{equation}
where
\begin{equation}\label{def0}
	\begin{array}{rllll}
		F&:=&\dis -a'(l(y(t)))z_v\int_{\Omega}z_w(\cdot,x)\dx-a'(l(y(t)))z_w\int_{\Omega}z_v(\cdot,\sigma)\d\sigma \\
		&&\dis -a''(l(y(t)))y\left(\int_{\Omega}z_v(\cdot,\sigma)\d\sigma\right)\left(\int_{\Omega}z_w(\cdot,\sigma)\d\sigma\right),
	\end{array}
\end{equation}
and $z_{v}:=G'(u)v,z_{w}:=G'(u)w$ are respective solutions to \eqref{diff1} with $v=v$ and $v=w$.
Moreover, for every $u\in L^\infty(\Omega)$, the linear mapping $v\mapsto z_v:= G'(u)v$ can be extended to a linear continuous mapping on $L^2(\omega_T)$ and there exists a constant $C:=C(\Omega,\|u\|_{L^\infty(\omega_T)},\|y^0\|_{L^\infty(\Omega)},T)>0$ such that
\begin{equation}\label{est1}
	\|z_v\|_{W(0,T)}\leq C\|v\|_{L^2(\omega_T)}.
\end{equation}	
\end{proposition}

\begin{proof}
Let $u,v,w\in L^\infty(\omega_T)$. It follows from Lemma \ref{lemmeG} that the control-to-state mapping $G:L^\infty(\omega_T)\to\mathbb{ X}, u\mapsto y$ is of class $\mathcal{C}^{2}$. Therefore, $G'(u)v$ and $G''(u)[v,w]$ exist. After straightforward computations, \eqref{diff1} and \eqref{diff2} follow easily. On one of the hand, since $v\in L^\infty(\omega_T)$, we use  Theorem \ref{theoremexistence} and Theorem \ref{theorembound0} to deduce that the system \eqref{diff1} admits a unique weak solution $z\in \mathbb{X}$. On the other hand thanks to Assumption \ref{ass1}, we obtain that $F$ defined in \eqref{def0} belongs to $L^\infty(Q)$. Therefore, the problem \eqref{diff2} has also a unique weak solution $z\in \mathbb{X}.$ Next, we note that if $v\in L^2(\omega_T)$, then \eqref{diff1} still have a unique weak solution $z\in W(0,T)$. Applying Theorem \ref{theoremexistence} with $\tilde{a}_0=a'(l(y(t)))y$ and thanks to \eqref{as1} and \eqref{estimationbound0}, we deduce that 
\begin{equation}\label{04}
	\begin{array}{rllll}
		\dis  \|z\|_{W(0,T)}\leq C\|v\|_{L^2(\omega_T)},
	\end{array}
\end{equation}	
with $C=C(\Omega,\|u\|_{L^\infty(\omega_T)},\|y^0\|_{L^\infty(\Omega)},T)$. This completes the proof.
\end{proof}

\subsection{The adjoint equation}
Next, let us introduce the adjoint state $q$ as the unique weak solution of the adjoint equation
\begin{equation}\label{adjoint}
\left\{
\begin{array}{rllll}
	\dis -q_t-\Delta q+a'(l(y(t)))\int_{\Omega}y(\cdot,\theta)q(\cdot,\theta)\d \theta+a(l(y(t)))q &=& y(u)-y_d &\mbox{in}& Q,\\
	\dis \partial_{\nu}q&=&0  &\mbox{in}& \Sigma,\\
	q(T,\cdot)&=& 0&\mbox{in}& \Omega.
\end{array}
\right.
\end{equation}
\begin{proposition}\label{existenceadjoint}
Let $u\in L^\infty(\omega_T)$ and $y_d\in L^\infty(Q)$. Let Assumption \ref{ass1} hold. Then there exists a unique weak solution $q\in \X$ to \eqref{adjoint}. Moreover, there exists a constant $C>0$ such that
\begin{equation}\label{estadjoint}
	\|q\|_{\C([0,T];L^2(\Omega))}+ \|q\|_{L^2(0,T;H^1(\Omega))} \leq C\left(\|y^0\|_{L^2(\Omega)}+\|y_d\|_{L^2(Q)}+\|u\|_{L^2(\omega_T)}\right)
\end{equation}
and
\begin{equation}\label{estadjointb}
	\|q\|_{L^\infty(Q)}\leq C\left(\|y^0\|_{L^\infty(\Omega)}+\|y_d\|_{L^\infty(Q)}+\|u\|_{L^\infty(\omega_T)}\right).
\end{equation}
\end{proposition}
\begin{proof} We make a change of variable $t\to q(T-t,x)$ and using the same arguments as in Theorem \ref{theoremexistence} , we deduce the existence of a unique weak solution $q\in W(0,T)$ to \eqref{adjoint} and the estimate \eqref{estadjoint} follows. Since $y,y_d\in L^\infty(Q)$, we obtain using similar arguments as in Theorem \ref{theorembound0} that $q\in L^\infty(Q)$ and thanks to \eqref{estimationbound0}, we deduce \eqref{estadjointb}.	
\end{proof}

\begin{remark}
As is Lemma \ref{lemmeG}, one can prove using the Assumption  \ref{ass1} that the mapping $u\mapsto q$, where $q$ is the solution to the adjoint state is of class $\mathcal{C}^1.$
\end{remark}
In the following, we show that the control-to-state operator and the map $u\mapsto q(u)$ are Lipschitz continuous.
\begin{proposition}\label{prop2} The control-to-state mapping  $u\mapsto G(u)$ is a Lipschitz continuous function from $L^2(\omega_T)$ onto $L^2(0,T;H^1(\Omega))$. In addition, the mapping $u\mapsto q(u)$, where $q$ is the solution to \eqref{adjoint} is also Lipschitz continuous. More precisely, for all $u_1,u_2\in L^\infty(\omega_T)$, there is a constant $C=C(\|u_2\|_{L^\infty(\omega_T)},\|y^0\|_{L^\infty(\Omega)},a_0,\Omega,T)>0$ such that the following estimates hold
\begin{equation}\label{estim20}
	\|G(u_1)-G(u_2)\|_{L^2(0,T;H^1(\Omega))}\leq C\|u_1-u_2\|_{L^2(\omega_T)}
\end{equation}	
and
\begin{equation}\label{estim21}
	\|q(u_1)-q(u_2)\|_{L^2(0,T;H^1(\Omega))}\leq C\|u_1-u_2\|_{L^2(\omega_T)}.
\end{equation}
\end{proposition}
\begin{proof}
Let $u_1,u_2\in L^\infty(\omega_T) .$
Set $z:=e^{-rt}(y_1-y_2)$ with  $y_1:=y(u_1)$, $y_2:=y(u_2)$ being solutions of \eqref{model} with $v=u_1$ and $v=u_2$ respectively and for some $r>0$. Then $z$ satisfies the following system
\begin{equation}\label{a1}
	\left\{
	\begin{array}{rllll}
		\dis z_{t}-\Delta z+a\left(l(y_1(t))\right)z+rz &=&e^{-rt}\left[(u_1-u_2)\chi_{\omega}- (a\left(l(y_1(t))\right)-a\left(l(y_2(t))\right))y_2\right]& \mbox{in}& Q,\\
		\dis  \partial_\nu z&=&0& \mbox{in}& \Sigma, \\
		\dis  z(0,\cdot)&=&0 &\mbox{in}&\Omega.
	\end{array}
	\right.
\end{equation}
Multiplying the first equation in \eqref{a1} by $z$, integrating by parts over $Q$ and using \eqref{as1}, we arrive to
\begin{equation}\label{ajout0}
	\begin{array}{rllll}
		\|\nabla z\|^2_{L^2(Q)}+ (r+a_0)\|z\|^2_{L^2(Q)} &\leq& \|a(l(y_1(t)))-a(l(y_2(t)))\|_{L^2(Q)}\|y_2\|_{L^\infty(Q)}\|z\|_{L^2(Q)}\\
		&&+ \|u_1-u_2\|_{L^2(\omega_T)}\|z\|_{L^2(Q)}.
	\end{array}
\end{equation}
Observing that $y^0\in L^\infty(\Omega)$, then using \eqref{estimationbound}, we have that
\begin{equation}\label{ajout4}
	\|y_2\|_{L^\infty(Q)}\leq C(\|u_2\|_{L^\infty(\omega_T)}, \|y^0\|_{L^\infty(\Omega)}).
\end{equation}
Moreover, using the mean value theorem and thanks to  \eqref{as1}, we have that
\begin{equation}\label{ajout1}
	\| a\left(l(y_1(t))\right)-a\left(l(y_2(t))\right)\|_{L^2(Q)}\leq Me^{rT}|\Omega|\|z\|_{L^2(Q)}.
\end{equation}

Now, combining \eqref{ajout0}, \eqref{ajout4}, \eqref{ajout1} and thanks to the H\"{o}lder's inequality, we arrive to
\begin{equation}\label{ajout2}
	\begin{array}{rllll}
		\|\nabla z\|^2_{L^2(Q)}+ (r+a_0)\|z\|^2_{L^2(Q)} &\leq& C(\|u_2\|_{L^\infty(\omega_T)},M,\Omega, \|y^0\|_{L^\infty(\Omega)},T)\|z\|^2_{L^2(Q)}\\
		&&\dis + \frac{1}{2}\|u_1-u_2\|^2_{L^2(\omega_T)}+\frac{1}{2} \|z\|^2_{L^2(Q)}.
	\end{array}
\end{equation}
Choosing $\dis r=C(\|u_2\|_{L^\infty(\omega_T)},M,\Omega, \|y^0\|_{L^\infty(\Omega)},T)+1/2$, we obtain
\begin{equation}\label{ajout3}
	\begin{array}{rllll}
		\min(1,a_0)\left(\|\nabla z\|^2_{L^2(Q)}+ \|z\|^2_{L^2(Q)}\right) &\leq&\dis   \frac{1}{2}\|u_1-u_2\|^2_{L^2(\omega_T)}.
	\end{array}
\end{equation}
Hence, \eqref{estim20} follows. Next, using similar arguments as above and thanks to \eqref{estim20}, we deduce \eqref{estim21}. This completes the proof.
\end{proof}
\begin{remark}
We note that if we take $u_1,u_2\in \mathcal{U}_{ad}$ in the Proposition \ref{prop2}, then the constant $C$ defined in the Proposition \ref{prop2} is such that $C:=C(\alpha,\beta,\|y^0\|_{L^\infty(\Omega)},\Omega,T)$.
\end{remark}
Let us define the reduced cost functional as
\begin{equation}\label{defj}
\mathcal{J}(v):=J(G(v),v).
\end{equation}
\begin{proposition}[\bf Twice Fr\'echet differentiability of $J$]\label{diff4}
Let $y$ be the solution of \eqref{model}. Let Assumption \ref{ass1} hold. Under the hypothesis of Lemma \ref{lemmeG}, the functional $\mathcal{J}:L^\infty(\omega_T)\to \R$ defined in \eqref{defj} is twice continuously Fr\'echet differentiable and for every $u, v,w\in L^\infty(\omega_T)$, we have
\begin{equation}\label{diff5}
	\mathcal{J}'(u)v=\int_{\omega_T}\left( q+\mu u\right)v\,\dq,
\end{equation}
and
\begin{equation}\label{diff6}
	\left.
	\begin{array}{lllll}
		\mathcal{J}''(u)[v,w] &=&\dis \int_{Q}Fq\,\dq+ \int_{Q} z_vz_w\,\dq + \mu \int_{\omega_T}vw\,\dq.
	\end{array}
	\right.
\end{equation}	
where $z_{v}:=G'(u)v,\, z_{w}:=G'(u)w$ are respective solutions to \eqref{diff1} with $v=v$ and $v=w$, $F$ is defined in \eqref{def0} and $q$ is the unique weak solution to the adjoint equation \eqref{adjoint}.
\end{proposition}
\begin{proof}
First, by the chain rule, we have that  $\mathcal{J}$ is twice continuously Fr\'echet differentiable, since by Lemma \ref{lemmeG}, $G$ has this property.\par 		
Let $u,v,w\in L^\infty(\omega_T)$.  After some straightforward calculations, we get
\begin{equation}\label{e1}
	\mathcal{J}'(u)v=\int_{Q} z_v(y(u)-y_d)\,\dq+ \mu \int_{\omega_T}uv\,\dq.
\end{equation}
Now, if we multiply the first equation of \eqref{diff1} by $q$ solution to the adjoint state \eqref{adjoint} and we integrate by parts over $Q$, we arrive to
\begin{equation}\label{e2}
	\int_{Q} z_v(y(u)-y_d)\,\dq=\int_{\omega_T}qv\,\dq.
\end{equation}
Combining \eqref{e1} and \eqref{e2} leads us to \eqref{diff5}. On the other hand, after some calculations, we obtain
\begin{equation}\label{e3}
	\left.
	\begin{array}{lllll}
		\mathcal{J}''(u)[v,w] &=&\dis \int_{Q}z_{vw}(y(u)-y_d)\,\dq + \int_{Q} z_vz_w\,\dq+\mu\int_{\omega_T} wv\,\dq.
	\end{array}
	\right.
\end{equation}	
Multiplying the equation \eqref{diff2} by $q$ solution to the adjoint state \eqref{adjoint} and after an integration by parts over $\Omega$, we get
\begin{equation}\label{e4}
	\left.
	\begin{array}{lllll}
		\dis \dis \int_{Q}z_{vw}(y(u)-y_d)\,\dq=\int_{Q}Fq\,\dq.
	\end{array}
	\right.
\end{equation}	
Therefore, combining \eqref{e3} and \eqref{e4}, we deduce \eqref{diff6}. The proof is finished.
\end{proof}

\subsection{First order necessary optimality conditions}
The aim of this section is to derive the first-order necessary optimality conditions and to characterize the optimal control. We have the following result.
\begin{theorem}\label{theoSO}
Let $u\in\mathcal U_{ad}$ be an $L^\infty$-local minimum for \eqref{opt}. Assume that Assumptions \ref{ass1} hold. Then,
\begin{equation}\label{ineq}
	\mathcal{J}'(u)(v-u)\geq 0\;\;\;\text{for every}\;\;\; v\in \U_{ad},
\end{equation}
equivalently	
\begin{equation}\label{ineq1}
	\int_{\omega_T}(q+\mu u)(v-u)\,\dx\geq 0\;\;\;\text{for every}\;\;\; v\in \U_{ad},
\end{equation}
where $q$ is the unique weak solution to \eqref{adjoint}.
\end{theorem}
\begin{proof}
Let $v\in \U_{ad}$ be arbitrary. Since $\U_{ad}$ is convex,  we have that $u+\lambda(v-u)\in \U_{ad}$ for all $\lambda\in (0,1]$. But $u$ is an $L^\infty$-local minimum so, $\mathcal{J}(u+\lambda(v-u))\geq \mathcal{J}(u)$. Hence,
$\dis\frac{\mathcal{J}(u+\lambda(v-u))-\mathcal{J}(u)}{\lambda}\geq 0$ for all $\lambda\in (0,1].$
Letting $\lambda\downarrow 0$ in the latter inequality, we obtain 	\eqref{ineq}. Thanks to \eqref{diff5}, we obtain
$$
\dis\int_{\omega_T}(q+\mu u)(v-u)\,\dx\geq 0\;\;\;\text{for every}\;\;\; v\in \U_{ad}.$$
This completes the proof.
\end{proof}
\begin{remark}
\begin{enumerate}
	\item[ $ $]
	\item 	We notice that any $L^\infty$-local solution of the variational inequality \eqref{ineq} is also a $L^2$-local one.
	\item 	The condition \eqref{ineq} can be rewritten as follows:  for a.e. $(x,t)\in \omega_T$,
	\begin{equation}\label{contr1vac}
		\dis \left\{\begin{array}{rlllll}
			u(x,t)=\alpha&\hbox{if}&\mu u(x,t)+q(x,t)>0,\\
			u(x,t)\in [\alpha,\beta] &\hbox{if}&\mu u(x,t)+q(x,t)=0,\\
			u(x,t)=\beta&\hbox{if}&\mu u(x,t)+q(x,t)<0,
		\end{array}
		\right.
	\end{equation}
	equivalently
	\begin{equation}\label{contr1}
		u=\min\left(\max\left(\alpha,-\frac{q}{\mu}\right),\beta\right) \text{ in }\omega_T.
	\end{equation}
\end{enumerate}
\end{remark}	

\begin{lemma}\label{cont1} Let $u\in L^\infty(\omega_T)$. Let Assumption \ref{ass1} hold. Then for every $u\in L^\infty(\omega_T)$,  the linear mapping $v\mapsto \mathcal{J}'(u)v$ can be extended to a linear continuous mapping $\mathcal{J}'(u):L^2(\omega_T)\to \R$ given by \eqref{diff5}. Moreover the bilinear  mapping $(v,w)\mapsto \mathcal{J}''(u)[v,w]$ can be extended to a bilinear continuous mapping on $\mathcal{J}''(u):L^2(\omega_T)\times L^2(\omega_T)\to \R$ given by \eqref{diff6}.
	\end{lemma}	
	\begin{proof} Let $u\in L^\infty(\omega_T)$ and $v\in L^2(\omega_T)$. From \eqref{diff5}, we have that
$$
\mathcal{J}'(u)v=\int_{\omega_T}\left( q+\mu u\right)v\,\dq,
$$
where  $q$ is the solution of \eqref{adjoint}.
Using \eqref{estadjoint},  we have that there is a constant $C>0$ independent of $v$ such that
$$
\begin{array}{lll}
	\dis |\mathcal{J}'(u)v|\leq C\|v\|_{L^2(\omega_T)}.
\end{array}
$$
Thus, the mapping $v\mapsto \mathcal{J}'(u)v$ is linear and continuous on $L^2(\omega_T)$. Next,
let and $v,w\in L^2(\omega_T)$. From \eqref{diff6}, we have

\begin{equation}\label{e01}
	\begin{array}{lll}
		\dis |\mathcal{J}''(u)[v,w]|&\leq& \|F\|_{L^2(Q)}\|q\|_{L^2(Q)}+\|z_v\|_{L^2(Q)}\|z_w\|_{L^2(Q)}+\mu \|v\|_{L^2(\omega_T)}\|w\|_{L^2(\omega_T)}.
	\end{array}
\end{equation}
Using the H\"{o}lder's inequality and \eqref{estimationbound}, we have 	
\begin{equation}\label{e02}
	\begin{array}{lll}
		\dis  \|F\|_{L^2(Q)}&\leq& 2M|\Omega|\|z_v\|_{L^2(Q)}\|z_w\|_{L^2(Q)}+M\|y\|_{L^\infty(Q)}|\Omega|^2\|z_v\|_{L^2(Q)}\|z_w\|_{L^2(Q)}\\
		&\leq& C(M,\|u\|_{L^\infty(\omega_T)},\|y^0\|_{L^\infty(\Omega)},\Omega)\|z_v\|_{L^2(Q)}\|z_w\|_{L^2(Q)}.
	\end{array}
\end{equation}	
Therefore, combining \eqref{est1}, \eqref{e01} and \eqref{e02}, we deduce that, there is a constant $C>0$ independent of $v$ and $w$ such that
$$\dis |\mathcal{J}''(u)[v,w]|\leq C \|v\|_{L^2(\omega_T)}\|w\|_{L^2(\omega_T)}.$$
Hence, the 	mapping $(v,w)\mapsto\mathcal{J}''(u)[v,w]$ is a bilinear continuous mapping on $L^2(\omega_T)\times L^2(\omega_T)$. This completes the proof.
\end{proof}
\section{Second order optimality conditions}\label{second}
\subsection{Second order necessary optimality conditions}\label{sufficientoptcond}
Note that the cost functional $\mathcal{J}$ associated to the optimization problem \eqref{opt} is non-convex and the first order optimality conditions given in Theorem \ref{theoSO} are necessary but not sufficient for optimality.

Before establishing the second-order necessary and sufficient conditions, we introduce some preliminary concepts retrieved from \cite{aronna2021,casas2012}.

\begin{definition}
$ $
\begin{enumerate}[-]
	\item The set of feasible directions is the set $S(u)$ defined by
	\begin{equation}\label{fset}
		S(u):=\{v\in L^\infty(\omega_T) : v=\lambda(w-u)\quad\text{for some}\;\;\lambda>0\;\;\text{and}\;\;w\in \mathcal{U}_{ad}\}.
	\end{equation}
	\item The critical cone is the set $C(u)$ defined by
	\begin{equation}\label{crcone}
		C(u):=\text{cl}_{L^2(\omega_T)}(S(u))\cap \{v\in L^2(\omega_T) : \mathcal{J}'(u)v=0\}.
	\end{equation}
	\item We also introduce the set $D(u)$ defined by
	\begin{equation}\label{du}
		D(u):=\{v\in S(u) : \mathcal{J}'(u)v=0\}.
	\end{equation}
\end{enumerate}
\end{definition}
We have the following result giving a characterization of the critical cone $C(u)$ defined in \eqref{crcone}. The proof follows the lines of the proof in \cite[Page 18]{aronna2021} (see also \cite[page 273]{casas2012}), so we omit it.
\begin{proposition}
Consider $v$ defined as follows:
for a.e $(x,t)\in \omega_T$,
\begin{equation}\label{eq1}
	v(x,t)\left\{
	\begin{array}{lllll}
		\dis  \geq 0 &if&  u(x,t)=\alpha,\\
		\dis \leq 0&if& u(x,t)=\beta,\\
		0&if&\mu u(x,t)+q(x,t)\neq 0,
	\end{array}
	\right.
\end{equation}
where $q$ is solution to \eqref{adjoint}.
The critical cone  $C(u)$ defined in \eqref{crcone} can be rewritten as 	
\begin{equation}\label{ccone}
	C(u)= \{v\in L^2(\omega_T) : v \;\text{fulfills}\; \eqref{eq1}\}.
\end{equation}
\end{proposition}

In the rest of the paper, we will adopt the following notation $\mathcal{J}''(u)v^2:=\mathcal{J}''(u)[v,v]$.

\begin{proposition}[Second order necessary optimality conditions]
Let $u\in \mathcal{U}$ be a $L^\infty$-local solution of system \eqref{opt2}. Then $\mathcal{J}''(u)v^2\geq 0$, for all $v\in C(u)$.
\end{proposition}

\begin{proof}
The proof is classical and can be found in \cite[pp. 246]{fredi2010} or \cite{kdz2022,kmw2022}.
\end{proof}

\begin{theorem}\label{hypothesis}
Let $u\in \mathcal{U}_{ad}$ be a control satisfying the first order optimality conditions \eqref{ineq}. Then the following hold:
\begin{enumerate}[(1)]
	\item The functional $\mathcal{J}:L^\infty(\omega_T)\to \R$ is of class $\mathcal{C}^2$. Furthermore, for every $u\in \mathcal{U}_{ad}$, there exists continuous extensions
	\begin{equation}\label{h1}
		\mathcal{J}'(u)\in \mathcal{L}(L^2(\omega_T),\R)\;\;\;\;\text{and}\;\;\; \mathcal{J}''(u)\in \mathcal{B}(L^2(\omega_T),\R).
	\end{equation}
	\item For any sequence $\left\{(u_k,v_k)\right\}_{k=1}^{\infty}\subset \mathcal{U}\times L^2(\omega_T)$ with $\|u_k-u\|_{L^2(\omega_T)}\to 0$ and $v_k \rightharpoonup v$ weakly in $L^2(\omega_T)$,
	\begin{equation}\label{h2}
		\dis	\mathcal{J}'(u)v=\lim_{k\to \infty} \mathcal{J}'(u_k)v_k,
	\end{equation}
	
	\begin{equation}\label{h3}
		\dis	\mathcal{J}''(u)v^2\leq \liminf_{k\to \infty} \mathcal{J}''(u_k)v_k^2,
	\end{equation}
	
	\begin{equation}\label{h4}
		\dis	\text{If}\; v=0, \;\; \text{then} \;\;\;\Lambda \liminf_{k\to \infty}\|v_k\|^2_{L^2(\omega_T)}\leq\liminf_{k\to \infty} \mathcal{J}''(u_k)v_k^2,
	\end{equation}
	for some $\Lambda>0$.
\end{enumerate}
\end{theorem}

\begin{proof}
The point $\textit{(1)}$ of Theorem \ref{hypothesis} follows from Proposition \ref{diff4} and Lemma \ref{cont1}.\\
We make the proof of the point $\textit{(2)}$ in three steps.\\
Let $\left\{(u_k,v_k)\right\}_{k=1}^{\infty}\subset \U_{ad}\times L^2(\omega_T)$ be a sequence such that $\|u_k-u\|_{L^2(\omega_T)}\to 0$ and $v_k \rightharpoonup v$ weakly in $L^2(\omega_T)$.\\
Let $y_k=:y(u_k)$ and $y:=y(u)$ be respectively the solutions of \eqref{model} with $v=u_k$ and $v=u$.
Then, using the Lipschitz continuity of the control-to-state operator and  the mapping $u\mapsto q(u)$ given in Proposition \ref{prop2}, we have that
\begin{equation}\label{k11}
	y_k\to y   \text{ strongly in }  L^2(0,T;H^1(\Omega))
\end{equation}
and
\begin{equation}\label{k12}
	q(u_k)\to q(u)   \text{ strongly in }  L^2(0,T;H^1(\Omega)).
\end{equation}
\noindent \textbf{Step 1.} We show \eqref{h2}.  Using the convergence \eqref{k12} and the strong convergence of $(u_k)_k$, it follows that  $\mu u_k+ q(u_k)\to \mu u+q(u)$ in $L^2(Q)$ as $k\to \infty$. By the expression of $\mathcal{ J}'$ given in \eqref{diff5}, we deduce that:
\begin{eqnarray*}
	\dis \lim_{k\to \infty} \mathcal{J}'(u_k)v_k=\lim_{k\to \infty}\int_{\omega_T}(\mu u_k+q(u_k))v_k\, \dq =\int_{\omega_T}(\mu u+q(u))v\, \dq= \mathcal{J}'(u)v.
\end{eqnarray*}
Thus \eqref{h2} is proved.\\
\noindent \textbf{Step 2.} We show \eqref{h3}. We write
\begin{equation}\label{d}
	\begin{array}{lllll}
		\dis 	\mathcal{J}''(u_k)v_k^2&=&\dis -2\int_{Q}a'(l(y_k(t)))q(u_k)z_{v_k}\left(\int_{\Omega}z_{v_k}(t,\theta)\,\d \theta\right)\dq\\
		&&\dis -\int_{Q}a''(l(y_k(t)))y_kq(u_k)\left(\int_{\Omega}z_{v_k}(t,\theta)\,\d \theta\right)^2\dq\\
		&&\dis +\int_{Q}z^2_{v_k}\, \dq+\mu\int_{\omega_T}|v_k|^2\, \dq.
	\end{array}
\end{equation}
$z_{v_k}$ is the unique weak solution of \eqref{diff1} with $u=u_k$ and $v=v_k$. On one of the hand, since the sequence $(u_k)_k$ is bounded in $L^2(\omega_T)$, we obtain from estimation \eqref{est1} , that the sequence $(z_{v_k})_k$ is  bounded in $W(0,T)$. Thus, up to a subsequence and as $k\to \infty$,  $(z_{v_k})_k$ converges weakly to $z_{v}$ in $W(0,T)$. Now, thanks to the compact embedding $W(0,T)\hookrightarrow L^2(Q)$ , we obtain that as $k\to \infty$
\begin{equation}\label{k10}
	z_{v_k}\to z_{v}   \text{ strongly in }  L^2(Q).
\end{equation}
On the other hand, we claim that as $k\to \infty$,
\begin{equation}\label{d0}
	\begin{array}{lllll}
		a'(l(y_k(t)))q(u_k)z_{v_k} \quad \text{converges strongly to }\quad a'(l(y(t)))q(u)z_{v}\quad \text{in}\quad L^2(Q)
	\end{array}
\end{equation}
and
\begin{equation}\label{d01}
	\begin{array}{lllll}
		a''(l(y_k(t)))q(u_k)y_{k} \quad \text{converges strongly to }\quad a''(l(y(t)))q(u)y\quad \text{in}\quad L^2(Q).
	\end{array}
\end{equation}
Indeed, using \eqref{as1}, we can show as in Theorem \ref{thmimp} that
\begin{equation}\label{k1}
	a'(l(y_k(\cdot)))\to a'(l(y(\cdot)))   \text{ strongly in }  L^2(Q)
\end{equation}
and
\begin{equation}\label{k2}
	a''(l(y_k(\cdot)))\to a''(l(y(\cdot)))   \text{ strongly in }  L^2(Q).
\end{equation}
Moreover, we  have
\begin{equation*}
	\begin{array}{lllll}
		\smallskip
		\|a'(l(y_k(t)))q(u_k)z_{v_k} - a'(l(y(t)))q(u)z_{v}\|_{L^2(Q)}\\
		\smallskip
		\leq \|a'(l(y_k(t))) - a'(l(y(t)))\|_{L^2(Q)}\|q(u_k)\|_{L^\infty(Q)}
		\smallskip \|z_{v_k}\|_{L^2(Q)}\\
		+M\|q(u_k)-q(u)\|_{L^2(Q)}\|z_{v_k}\|_{L^2(Q)}\\
		+M\|z_{v_k} -z_{v}\|_{L^2(Q)}\|q(u)\|_{L^\infty(Q)}.
	\end{array}
\end{equation*}
Passing to the limit as $k\to \infty$ in this latter inequality while using \eqref{k12}, \eqref{k10} and \eqref{k1} leads us to the claim \eqref{k1}. Arguing as the same, while using the convergences \eqref{k11}, \eqref{k12} and \eqref{k2}, we obtain \eqref{d01}.
Taking the limit as $k\to \infty$ in \eqref{d} and using the lower-semi continuity of the $L^2$-norm and the above convergences,  we deduce that
\begin{eqnarray*}
	\dis \lim_{k\to \infty}  \mathcal{J}''(u_k)v_k^2
	&\geq&\dis -2\lim_{k\to \infty}\int_{Q}a'(l(y_k(t)))q(u_k)z_{v_k}\left(\int_{\Omega}z_{v_k}(t,\theta)\,\d \theta\right)\dq\\
	&&\dis -\liminf_{k\to \infty}\int_{Q}a''(l(y_k(t)))y_kq(u_k)\left(\int_{\Omega}z_{v_k}(t,\theta)\,\d \theta\right)^2\dq\\
	&&\dis +\lim_{k\to \infty} \int_{Q}z^2_{v_k}\, \dq+\mu \liminf_{k\to \infty}\int_{\omega_T}|v_k|^2\, \dq\\
	&\geq&\dis -2\int_{Q}a'(l(y(t)))q(u)z_{v}\left(\int_{\Omega}z_{v}(t,\theta)\,\d \theta\right)\dq\\
	&&\dis -\int_{Q}a''(l(y(t)))yq(u)\left(\int_{\Omega}z_{v}(t,\theta)\,\d \theta\right)^2\dq\\
	&&\dis +\int_{Q}z^2_{v}\, \dq+\mu \int_{\omega_T}|v|^2\, \dq\\
	&=& \mathcal{J}''(u)v^2.
\end{eqnarray*}
Hence \eqref{h3} is obtained.

\noindent \textbf{Step 3.} Finally, we show \eqref{h4}.
If $v=0$, then in \eqref{d}, the first, the second and the third terms tend to $0$, except the last. Hence,
$$\Lambda  \liminf_{k\to \infty}\|v_k\|^2_{L^2(\omega_T)}\leq   \lim_{k\to \infty}  \mathcal{J}''(u_k)v_k^2,$$
with $\Lambda=\mu$. This completes the proof.
\end{proof}
We have  the following result giving second-order sufficient conditions for locally optimal solutions.	
\begin{theorem}\label{Quadratic growth}
Let $u\in \mathcal{U}_{ad}$ satisfy \eqref{ineq} and
\begin{equation}\label{cdt1}
	\mathcal{J}''(u)v^2> 0 \quad \forall v\in C(u)\backslash\{0\}.
\end{equation}
Then there are two constants $\gamma>0$ and $\eta>0$ such that the quadratic growth condition
\begin{equation}\label{cdt2}
	\mathcal{J}(v)\geq \mathcal{J}(u)+\frac{\eta}{2}\|v-u\|^2_{L^2(\omega_T)},\quad \forall v\in \mathcal{U}_{ad}\cap B^2_{\gamma}(u),
\end{equation}
holds. Hence $u$ is locally optimal in the sense of $L^2(\omega_T)$.
\end{theorem}
\begin{proof}
From Theorem \ref{hypothesis}, the assumptions of  \cite[Theorem 2.3, page 265]{casas2012} are fulfilled. Hence, \eqref{cdt2} holds.
\end{proof}

\section{Concluding remarks}
\label{conclusion}
In this paper, we studied the optimal control of a nonlinear parabolic system. In contrast to many studies, the  nonlinearity of the system \eqref{model} involves an nonlocal term and so the problem is more challenging. We proved rigorously some regularity results of systems considered and established the existence as well as the characterization of the optimal solutions. Because of the nonlinearity, the functional defined in \eqref{opt1} was non convex and the first order optimality conditions was necessary but not sufficient for the optimality. We then derived the second optimality conditions and obtained by the way, the local uniqueness of the optimal solutions. As a perspective, one could consider nonlocal problems with a control appearing nonlinearly in the state equation. More precisely, it means to consider, the right hand side of \eqref{model} on the form $F(x,y,v)$. This has been done recently in \cite{casas2020o} for an elliptic equation and in \cite{kmw2023frac} for class of semilinear fractional elliptic equations. One of the main novelty relies on the existence proof of optimal controls which differ from the standard one that is based on minimizing sequences of controls and their weak convergence. Here, the proof uses the measurable selection theorem \cite{aubin2009}. Another important task will be to derive the first- and second- order conditions for optimality.

\appendix
\renewcommand{\thesection}{A} 
\renewcommand{\theequation}{A.\arabic{equation}}    

\setcounter{equation}{0}  

\section*{Appendix A. Existence and uniqueness of solution to \eqref{generalin}}
By a change of variable $p=e^{-rt}y$, for a suitable $r>0$, which will be chosen later and for $y$ satisfying \eqref{generalin}. We consider the following problem:
\begin{equation}\label{generalin0}
\left\{
\begin{array}{rllll}
	\dis p_t-\Delta p+(r+a(l(z(t))))p+\tilde{a}_0\int_{\Omega}p(t,x)\dx &=& e^{-rt}f &\mbox{in}& Q,\\
	\dis \partial_{\nu}p&=&0  &\mbox{in}& \Sigma,\\
	p(0,\cdot)&=& y^0 &\mbox{in}& \Omega.
\end{array}
\right.
\end{equation}
We define the weak solution to the system \eqref{generalin0} as follows.
\begin{definition}\label{weakgenelin0}
Let $f,z\in L^2(Q)$, $\tilde{a}_0\in L^\infty(Q)$ and $y^{0}\in L^2(\Omega)$. Let $a(\cdot)$  satisfying the Assumption \ref{ass1}.  We say that a function $p\in L^2(0,T;H^1(\Omega))$  is a weak solution to \eqref{generalin0}, if the following equality
\begin{equation}\label{weakgenlin0}
	\begin{array}{lll}
		\dis \int_\Omega y^0\,\phi(0,\cdot)\, \dx+\int_{Q}e^{-rt}f\phi\,\dq&=&\dis -\int_{0}^{T} \left\langle \phi_t(t),p(t)\right\rangle_{(H^{1}(\Omega))^\star,H^1(\Omega)} \, \dt+ \int_Q \nabla \phi\cdot\nabla p\, \dq\\
		&+&\dis \int_Q( r+a\left(l(z(t))\right)) \phi p\, \dq+\dis \int_Q\left(\tilde{a}_0 \dis \int_{\Omega}p(t,\sigma)\d\sigma\right)\phi\, \dq\\
	\end{array}
\end{equation}
holds for every $\phi \in H(Q):=\{\varphi\in W(0,T): \phi(T,\cdot)=0\}$.
\end{definition}
We have the following result.

\begin{theorem}\label{theoremexistencelin0}
Let $f, z\in L^2(Q)$, $\tilde{a}_0\in L^\infty(Q)$ such that $\|\tilde{a}_0\|_{L^\infty(Q)}\leq \tilde{C}$, for some $\tilde{C}>0$ and $y^{0}\in L^2(\Omega)$. Let $a(\cdot)$ satisfying the Assumption \ref{ass1}. Then, there exists a unique weak solution $p \in W(0,T)$ to \eqref{generalin0} in the sense of Definition \ref{weakgenelin0}.
In addition, the following estimations hold

\begin{equation}\label{estimationgenlin0}
	\|p\|_{\C([0,T];L^2(\Omega))} \leq \left(\frac{1}{a_0}\|f\|^2_{L^2(Q)}+\|y^0\|^2_{L^2(\Omega)}\right)^{1/2},
\end{equation}
\begin{equation}\label{estimationgenlin1}
	\|p\|_{L^2(0,T;H^1(\Omega))} \leq \frac{1}{\sqrt{\min(1,a_0)}}\left(\frac{1}{a_0}\|f\|^2_{L^2(Q)}+\|y^0\|^2_{L^2(\Omega)}\right)^{1/2},
\end{equation}
and
\begin{equation} \label{estimationint3}
	\dis \|p_t\|_{L^2(0,T;(H^1(\Omega))^\star)}\leq \frac{1}{\sqrt{\min(1,a_0)}} \left(1+2\tilde{C}|\Omega|+a_1\right)\left(\frac{1}{a_0}\|f\|^2_{L^2(Q)}+\|y^0\|^2_{L^2(\Omega)}\right)^{1/2}+\|f\|_{L^2(Q)}.
\end{equation}
\end{theorem}
\begin{proof}We proceed in four steps.\par 	
\textbf{Step 1.}  We prove the existence by using Theorem 1.1 \cite[Page 37]{lions2013}.
Recall that the norm on  $L^2(0,T;H^1(\Omega))$ is given by
$$\|p\|^2_{L^2(0,T;H^1(\Omega))}=\int_0^T \|p(\cdot,t)\|^2_{H^1(\Omega)}\, dt.$$
We consider the norm defined on  $H(Q)$ by $\|p\|^2_{H(Q)}:=\|p\|^2_{L^2(0,T;H^1(\Omega))}+\|p(0,\cdot)\|^2_{L^2(\Omega)}.$
It is clear that we have the continuous embedding $H(Q)\hookrightarrow L^2(0,T;H^1(\Omega))$.\par 		
Now, let $\varphi \in H(Q)$ and consider the bilinear form $\mathcal{E}(\cdot,\cdot):L^2(0,T;H^1(\Omega))\times H(Q)\to\R$ given by
\begin{equation}\label{defCalE}
	\begin{array}{lll}
		\dis	\mathcal{E}(z,\phi)&=& \dis -\int_{0}^{T} \left\langle \phi_t(t),p(t)\right\rangle_{(H^{1}(\Omega))^\star,H^1(\Omega)} dt+ \int_Q \nabla \phi\nabla p\, \dq+\int_Q( r+a\left(l(z(t))\right)) \phi p\, \dq\\
		&&+\dis \int_Q\left(\tilde{a}_0 \dis \int_{\Omega}p(t,\sigma)\d\sigma\right)\phi\, \dq.
	\end{array}
\end{equation}
First, we note that
\begin{equation}\label{1}
	\left|\int_Q\left(\tilde{a}_0 \dis \int_{\Omega}p(t,\sigma)\d\sigma\right)\phi\, \dq\right|\leq \tilde{C}|\Omega|\|p\|_{L^2(Q)}\|\phi\|_{L^2(Q)}.
\end{equation}
Using Cauchy Schwarz's inequality and \eqref{as1},  we get that
\begin{align*}
	\dis |\mathcal{E}(p,\phi)|
	\leq  \Big[\|\phi_{t}\|_{L^2(0,T;(H^1(\Omega))^\star)}+\left(r+\tilde{C}|\Omega|+a_1\right)\|\phi\|_{L^2(Q)}+\|\phi\|_{L^2(0,T;H^1(\Omega))}\Big]\|p\|_{L^2(0,T;H^1(\Omega))}.
\end{align*}
Consequently, for every fixed $\phi\in H(Q),$
the functional  $p\mapsto \mathcal{E}(p,\phi)$ is continuous on $L^2(0,T;H^1(\Omega)).$\par 			
Next, thanks to \eqref{as1}, \eqref{1} and choosing $r=\tilde{C}|\Omega|$,  we have that for every  $\phi\in H(Q)$,
$$\begin{array}{rlll}
	\mathcal{E}(\phi,\phi)&=&\dis-\int_{0}^{T} \left\langle \phi_t(t),\phi(t)\right\rangle_{(H^{1}(\Omega))^\star,H^1(\Omega)} \, \dt+ \int_Q |\nabla \phi|^2\, \dq+\int_Q( r+a\left(l(z(t))\right)) \phi^2\, \dq\\
	&&+\dis \int_Q\left(\tilde{a}_0 \dis \int_{\Omega}\phi(t,\sigma)\d\sigma\right)\phi\, \dq\\
	&\geq& \dis \frac 12  \|\phi(0,\cdot)\|^2_{L^2(\Omega)}+ \min (1,a_0)\|\phi\|^2_{L^2(0,T;H^1(\Omega))}\\
	&\geq&\min (1/2,a_0) \|\phi\|^2_{H(Q)}.
\end{array}
$$
Hence, $\mathcal E$ is coercive on $H(Q)$.\par
Finally, let us consider the functional $L:H(Q)\to \R$ defined by
$$
L(\phi):=\dis \int_\Omega y^0\,\phi(0,\cdot)\, \dx+\int_{Q}e^{-rt}f\phi\,\dq.
$$
Using Cauchy Schwarz's inequality,  there is a constant $C>0$ such that
$$
|L(\phi)|
\leq  C\left( \|y^0\|_{L^2(\Omega)}+\|f\|_{L^2(Q)}\right)\|\phi\|_{H(Q)}.
$$
Therefore,  $L(\cdot)$ is continuous and linear on $H(Q)$. It follows from Theorem 1.1 \cite[Page 37]{lions2013} that there exists $p\in L^2(0,T;H^1(\Omega))$ such that
$\mathcal{E}(p,\phi)= L(\phi),\quad \forall \phi \in H(Q)$.
We have shown that the system \eqref{generalin0} has a solution $p\in L^2(0,T;H^1(\Omega))$ in the sense of Definition \ref{weakgenelin0}.\par	
\noindent\textbf{Step 2.} We show that $p_t\in L^2(0,T;(H^1(\Omega))^\star)$.\\
Since $p\in L^2(0,T;H^1(\Omega))$,  then
$\dis p_t=\Delta p-(r+a(l(z(t))))p-\tilde{a}_0\int_{\Omega}p(t,x)\dx+e^{-rt}f\in (H^1(\Omega))^\star$ for almost every $t\in (0,T)$. \\
If we take the duality map between the first equation in   \eqref{generalin0} and $\phi\in L^2((0,T);H^1(\Omega)),$
we arrive to
$$\begin{array}{rlll}
	\dis  \left\langle p_t(t),\phi(t)\right\rangle_{(H^1(\Omega))^\star,H^1(\Omega)}&=&\dis - \int_\Omega \nabla \phi\cdot\nabla p\, \dx-\int_\Omega( r+a\left(l(z(t))\right)) \phi p\, \dx\\
	&&-\dis \int_\Omega\left(\tilde{a}_0 \dis \int_{\Omega}p(t,\sigma)\d\sigma\right)\phi\, \dx+e^{-rt} \int_\Omega f(t)\phi(t)\,\dx.
\end{array}
$$
This implies that
$$
\begin{array}{llll}
	\dis |\left\langle p_t(t),\phi(t)\right\rangle_{(H^1(\Omega))^\star,H^1(\Omega)}| &\leq& \dis \dis \|\nabla p(t)\|_{L^2(\Omega)} \|\nabla \phi(t)\|_{L^2(\Omega)} + \left(2\tilde{C}|\Omega|+a_1\right)\|p(t)\|_{L^2(\Omega)}\|\phi(t)\|_{L^2(\Omega)}\\
	&&+\|f(t)\|_{L^2(\Omega)}\|\phi(t)\|_{L^2(\Omega)}
	.
\end{array}
$$
Integrating this inequality over $(0,T)$,  we can deduce that
\begin{equation} \label{estimationint3-1}
	\begin{array}{llll}
		\dis\int_0^T |\left\langle p_t(t),\phi(t)\right\rangle_{(H^1(\Omega))^\star,H^1(\Omega)}| \dt\\
		\leq \left[\left(1+2\tilde{C}|\Omega|+a_1\right)\|p\|_{L^2(0,T;H^1(\Omega))}+\|f\|_{L^2(Q)}\right]
		\|\phi\|_{L^2(0,T;H^1(\Omega))},
	\end{array}
\end{equation}
from which we obtain that
\begin{equation}\label{estcorpt}
	\dis \|p_t\|_{L^2(0,T;(H^1(\Omega))^\star)} \leq \left(1+2\tilde{C}|\Omega|+a_1\right)\|p\|_{L^2(0,T;H^1(\Omega))}+\|f\|_{L^2(Q)}.
\end{equation}
\noindent\textbf{Step 3.} We establish the estimates \eqref{estimationgenlin0}-\eqref{estimationint3}.
If we multiply the first equation in   \eqref{generalin0} by $p\in L^2((0,T);H^1(\Omega))$, then  use \eqref{1}, \eqref{as1}, the Young's inequality and the fact that $r=\tilde{C}|\Omega|$, we arrive to
$$\begin{array}{llll}
	\dis  \frac 12\frac{d}{dt}\|p(t)\|^2_{L^2(\Omega)}+\|p(t)\|^2_{H^1(\Omega)}+a_0\|p(t)\|^2_{L^2(\Omega)}
	\leq\frac{1}{2a_0}\|f(t)\|^2_{L^2(\Omega)}+\frac{a_0}{2}\|p(t)\|^2_{L^2(\Omega)}.
\end{array}
$$
Hence,
\begin{equation}\label{inter2}
	\frac{d}{dt}\|p(t)\|^2_{L^2(\Omega)}+\min(1,a_0)\|p(t)\|^2_{H^1(\Omega)}\leq \dis \frac{1}{a_0}\|f(t)\|^2_{L^2(\Omega)}.
\end{equation}
Integrating \eqref{inter2} over $(0,\tau)$, with $\tau\in [0,T]$, we get that
$$
\|p(\tau)\|^2_{L^2(\Omega)}+\min(1,a_0)\int_{0}^{\tau}\|p(t)\|^2_{H^1(\Omega)}\,dt\leq \frac{1}{a_0}\|f\|^2_{L^2(Q)}+\|y^0\|^2_{L^2(\Omega)},
$$
from which we deduce  \eqref{estimationgenlin0} and \eqref{estimationgenlin1}. Using \eqref{estimationgenlin1} in \eqref{estcorpt}, we obtain \eqref{estimationint3}.\par

\noindent \textbf{Step 4.} We prove the uniqueness of $p$ solution to \eqref{generalin0}.\\
Assume that there exist $p_1$ and $p_2$ solutions to \eqref{general} with the same right hand side $f$ and initial datum $y^0$.  Set $q:=e^{-rt}(p_1-p_2)$, with some $r>0$. Then  $q$ satisfies the following system
\begin{equation}\label{aa10}
	\left\{
	\begin{array}{rllll}
		\dis q_{t}-\Delta q+a\left(l(z(t))\right)q+\tilde{a}_0\int_{\Omega}q(t,x)\dx +rq &=&0& \mbox{in}& Q,\\
		\dis  \partial_\nu q&=&0& \mbox{in}& \Sigma, \\
		\dis  q(0,\cdot)&=&0 &\mbox{in}&\Omega.
	\end{array}
	\right.
\end{equation}
Multiplying the first equation in \eqref{aa1} by $q$, integrating by parts over $\Omega$, using \eqref{as1} and \eqref{1}, we arrive to
\begin{equation}\label{ajout010}
	\begin{array}{rllll}
		\dis \|q(T)\|^2_{L^2(\Omega)}+\|\nabla q\|^2_{L^2(Q)}+ (r+a_0-\tilde{C}|\Omega|)\|q\|^2_{L^2(Q)}
		\leq 0.
	\end{array}
\end{equation}

Now, choosing $r=\tilde{C}|\Omega|$ in \eqref{ajout010}, we deduce that
$q=0$ in $Q$. Hence $q_1=q_2$. This completes the proof.

\end{proof}
Now, coming back to  problem \eqref{generalin}, we have the following result.

\begin{corollary}\label{corollaryexistence}
Let $f, z\in L^2(Q)$, $\tilde{a}_0\in L^\infty(Q)$ such that $\|\tilde{a}_0\|_{L^\infty(Q)}\leq \tilde{C}$, for some $\tilde{C}>0$ and $y^{0}\in L^2(\Omega)$. Let $a(\cdot)$ satisfying the Assumption \ref{ass1}.  Then, there exists a unique weak solution $y \in W(0,T)$ to \eqref{generalin}.
In addition, the following estimates hold 	

\begin{subequations}\label{estimation22W0Tout}
	\begin{alignat}{11}
		\|y\|_{\C([0,T];L^2(\Omega))} \leq e^{\tilde{C}|\Omega|T}\left(\frac{1}{a_0}\|f\|^2_{L^2(Q)}+\|y^0\|^2_{L^2(\Omega)}\right)^{1/2} ,
		\label{estimation1S1}\\
		\|y\|_{L^2(0,T;H^1(\Omega))} \leq \frac{e^{\tilde{C}|\Omega|T}}{\sqrt{\min(1,a_0)}}\left(\frac{1}{a_0}\|f\|^2_{L^2(Q)}+\|y^0\|^2_{L^2(\Omega)}\right)^{1/2},
		\label{estimation1}
	\end{alignat}
	and
	\begin{equation}\label{estimation2}
		\begin{array}{llll}
			\dis \|y_t\|_{L^2(0,T;(H^1(\Omega))^\star)}\leq \frac{e^{2\tilde{C}|\Omega|T}}{\sqrt{\min(1,a_0)}}\left(1+3\tilde{C}|\Omega|+a_1\right)
			\left(\frac{1}{a_0}\|f\|^2_{L^2(Q)}+\|y^0\|^2_{L^2(\Omega)}\right)^{1/2}+\|f\|_{L^2(Q)}.
		\end{array}
	\end{equation}
\end{subequations}
\end{corollary}	
\begin{proof}
Since $y=e^{\tilde{C}|\Omega|t}p$ is a weak solution of \eqref{generalin} if and only if $p$ is a weak solution of \eqref{generalin0}, we have from Theorem \ref{theoremexistencelin0} that there exists a unique solution $y \in W(0,T)$ of \eqref{generalin}.\par
Next,  letting $p=e^{-\tilde{C}|\Omega|t}y$  in \eqref{estimationgenlin0} and \eqref{estimationgenlin1} we  respectively deduce that \eqref{estimation1S1} and \eqref{estimation1} hold true.\par
Finally,  replacing $p$ by $e^{-\tilde{C}|\Omega|t}y$ in \eqref{estimationint3-1}, we deduce that 
$$
\begin{array}{llll}
	\dis \int_0^T |\left\langle y_t(t),\phi(t)\right\rangle_{(H^1(\Omega))^\star,H^1(\Omega)}| \dt\\
	\leq  \left[e^{\tilde{C}|\Omega|T}\left(1+3\tilde{C}|\Omega|+a_1\right)\|p\|_{L^2(0,T;H^1(\Omega))}+\|f\|_{L^2(Q)}\right]
	\|\phi\|_{L^2(0,T;H^1(\Omega))},\,\forall \phi\in L^2(0,T;H^1(\Omega)),
\end{array}
$$
which in view of  \eqref{estimation1} implies that 
$$
\begin{array}{llll}
	\dis \|y_t\|_{L^2(0,T;(H^1(\Omega))^\star)} \leq
	\frac{e^{2\tilde{C}|\Omega|T}}{\sqrt{\min(1,a_0)}}\left(1+3\tilde{C}|\Omega|+a_1\right)
	\left(\frac{1}{a_0}\|f\|^2_{L^2(Q)}+\|y^0\|^2_{L^2(\Omega)}\right)^{1/2}+\|f\|_{L^2(Q)}.
\end{array}
$$

This completes the proof.
\end{proof}

\bibliographystyle{abbrv}
\bibliography{mybibfile}
\end{document}